\documentclass[12pt]{amsart}
\usepackage{amssymb}

\let\Bbb\mathbb

\def\>{\relax\ifmmode\mskip.666667\thinmuskip\relax\else\kern.111111em\fi}
\def\<{\relax\ifmmode\mskip-.333333\thinmuskip\relax\else\kern-.0555556em\fi}
\def\vsk#1>{\vskip#1\baselineskip}
\def\vv#1>{\vadjust{\vsk#1>}\ignorespaces}
\def\vvn#1>{\vadjust{\nobreak\vsk#1>\nobreak}\ignorespaces}
\def\vvgood{\vadjust{\penalty-500}} \let\alb\allowbreak
\def\fratop{\genfrac{}{}{0pt}1}
\def\satop#1#2{\fratop{\scriptstyle#1}{\scriptstyle#2}}
  \let\ssize\scriptstyle
\let\sssize\scriptscriptstyle \def\tfrac{\textstyle\frac}

\let\Medskip\medskip
\def\medskip{\par\Medskip}
\let\Bigskip\bigskip
\def\bigskip{\par\Bigskip}

\let\Maketitle\maketitle
\def\maketitle{\Maketitle\thispagestyle{empty}\let\maketitle\empty}

\newtheorem{thm}{Theorem}[section]
\newtheorem{cor}[thm]{Corollary}
\newtheorem{lem}[thm]{Lemma}
\newtheorem{prop}[thm]{Proposition}

\numberwithin{equation}{section}

\theoremstyle{definition}
\newtheorem*{rem}{Remark}

\let\mc\mathcal
\let\nc\newcommand

\let\dl\delta
\let\Dl\Delta

\let\epe\epsilon

\let\la\lambda

\let\pho\phi
\let\phi\varphi
\let\si\sigma
\let\Si\Sigma
\let\Sig\varSigma

\let\thi\vartheta

\let\om\omega
\let\Om\Omega

\let\der\partial

\let\ox\otimes

\let\ge\geqslant
\let\geq\geqslant
\let\le\leqslant
\let\leq\leqslant

\let\on\operatorname
\let\bi\bibitem
\let\bs\boldsymbol

\def\C{{\mathbb C}}
\def\Z{{\mathbb Z}}

\def\Ac{{\mc A}}
\def\B{{\mc B}}
\def\F{{\mc F}}
\def\Hc{{\mc H}}

\def\Oc{{\mc O}}

\def\V{{\mc V}}

\def\+#1{^{\{#1\}}}
\def\lsym#1{#1\alb\dots\relax#1\alb}
\def\lc{\lsym,}
\def\lox{\lsym\ox}

\def\ch{\on{ch}}

\def\End{\on{End}}

\def\qdet{\on{qdet}}
\def\rdet{\on{rdet}}

\def\tbigoplus{\mathop{\textstyle{\bigoplus}}\limits}

\def\Wr{\on{Wr}}

\def\ii{i,\<\>i}
\def\ij{i,\<\>j}
\def\ik{i,\<\>k}
\def\il{i,\<\>l}
\def\ji{j,\<\>i}

\def\jk{j,\<\>k}
\def\kj{k,\<\>j}
\def\kl{k,\<\>l}

\def\gln{\mathfrak{gl}_N}
\def\glN{\mathfrak{gl}_N}

\def\Ugln{U(\gln)}

\def\Yn{Y\<(\gln)}

\def\beq{\begin{equation}}
\def\eeq{\end{equation}}
\def\be{\begin{equation*}}
\def\ee{\end{equation*}}

\nc{\bea}{\begin{eqnarray*}}
\nc{\eea}{\end{eqnarray*}}
\nc{\bean}{\begin{eqnarray}}
\nc{\eean}{\end{eqnarray}}
\nc{\Ref}[1]{{\rm(\ref{#1})}}

\def\g{{\mathfrak g}}
\def\h{{\mathfrak h}}

\nc{\Il}{{\mc I_{\bs\la}}}
\nc{\bla}{{\bs\la}}
\nc{\Fla}{\F_\bla}
\nc{\tfl}{{T^*\Fla}}
\nc{\GL}{{GL_n(\C)}}
\nc{\GLC}{{GL_n(\C)\times\C^*}}

\def\Di{{\tfrac1D}}

\def\DVe{\Di\V^A}
\def\DLe{(\DVe)_\bla}
\def\DLs{\DLe^\sing}

\def\zzz{z_1\lc z_n}

\def\Czh{\C[\zzz]}

\def\Czs{\C[\zzz]^{\<\>S}}

\let\sd s 

\def\tti{\tilde t}
\def\ttbi{\tilde{\<\bs t\<\>}}
\def\vvbi{\tilde{\<\bs v\<\>}}

\def\ib{\bs i}
\def\jb{\bs j}

\def\iib{\ib,\<\>\ib}
\def\ijb{\ib,\<\>\jb}

\def\zb{\bs z}
\def\zzi{z_1\lc z_i,z_{i+1}\lc z_n}
\def\zzii{z_1\lc z_{i+1},z_i\lc z_n}

\def\Bck{\B^{\<\>\kk}}

\def\Hck{\Hc^{\<\>\kk}}

\def\ddk_#1{\kk_{#1}\<\>\frac\der{\der\<\>\kk_{#1}}}

\def\vone{v_1\<\lox v_1} \def\voxn{v_1^{\otimes n}}

\def\bul{\mathbin{\raise.2ex\hbox{$\sssize\bullet$}}}
\def\intt{\mathchoice
{\mathop{\raise.2ex\rlap{$\,\,\ssize\backslash$}{\intop}}\nolimits}
{\mathop{\raise.3ex\rlap{$\,\sssize\backslash$}{\intop}}\nolimits}
{\mathop{\raise.1ex\rlap{$\sssize\>\backslash$}{\intop}}\nolimits}
{\mathop{\rlap{$\sssize\<\>\backslash$}{\intop}}\nolimits}}

\let\kk q 
\let\cc c
\def\kkk{\kk_1\lc\kk_N}

\let\Ko K

\def\GZ/{Gelfand-Zetlin}
\def\KZ/{{\slshape KZ\/}}
\def\qKZ/{{\slshape qKZ\/}}
\def\XXX/{{\slshape XXX\/}}

\def\zz{{\bs z}}
\def\aa{{\bs a}}
\def\bb{{\bs b}}
\def\TT{{\bs t}}
\def\qq{{\bs q}}

\nc{\Xe}{\mathcal X}
\nc{\Tee}{\mathcal S}
\nc{\D}{{\mathcal {D}}}
\nc{\Te}{\mathfrak{T}}
\def\Bq{B^{\<\>\qq}}
\def\Bql{B^{\<\>\qq,\<\>\bla}}
\def\Bcq{\B^{\<\>\qq}}
\def\Bo{\B^{\<\>\qq=\bs1}}
\def\Cl{C^{\<\>\bla}}
\def\Fql{F^{\>\qq,\<\>\bla}}
\def\Oq{\Oc^{\<\>\qq}}
\def\Oql{\Oq_\bla}
\def\Ol{{\Oc_\bla}}
\def\Omq{\Om^{\<\>\qq}}
\def\Omql{\Omq_\bla}
\def\Oml{\Om_\bla}
\def\DOl{\D^{\>\Ol}}
\def\kmin{k_{\<\>\textrm{min}}}

\def\yy{{\<\>\raise.22ex\hbox{$\scriptscriptstyle Y\!\!\!$}}}
\def\nnla{{n,N\<,\<\>\bs\la}}

\def\O{{\mc O}}
\def\lba{{\bs\la\>,\aa}}

\def\sing{{\textit{sing}}}
\def\Gaud{{\textit{Gaudin}}}
\nc{\Se}{\mathfrak S}
\nc{\Dt}{\D_{\>\ttbi}(u,\tau)}
\nc{\Sym}{{\on{Sym}}}

\begin{document}

\hrule width0pt
\vsk->

\title[Spaces of quasi-exponentials and representations of
the Yangian $Y(\gln)$]
{Spaces of quasi-exponentials and representations of\\[2pt]
the Yangian $Y(\gln)$ }

\author[E.\,Mukhin, V.\,Tarasov, and A.\,Varchenko]
{E.\,Mukhin$\>^{*,1}$, V.\,Tarasov$\>^{\star,*,2}$,
and A.\,Varchenko$\>^{\diamond,3}$}

\thanks{${\>\kern-\parindent}^1\,$Supported in part by NSF grant DMS-0900984}
\thanks{${\>\kern-\parindent}^2\,$Supported in part by NSF grant DMS-0901616}

\thanks{${\>\kern-\parindent}^3\,$Supported in part by NSF grant DMS-1101508}

\maketitle

\begin{center}
{\it $^\star\<$Department of Mathematical Sciences
Indiana University\,--\>Purdue University Indianapolis\\
402 North Blackford St, Indianapolis, IN 46202-3216, USA\/}

\medskip
{\it $^*\<$St.\,Petersburg Branch of Steklov Mathematical Institute\\
Fontanka 27, St.\,Petersburg, 191023, Russia\/}

\medskip
{\it $^\diamond\<$Department of Mathematics, University of North Carolina
at Chapel Hill\\ Chapel Hill, NC 27599-3250, USA\/}
\end{center}

\medskip

\begin{abstract}
We consider a tensor product $V(\bb)= \otimes_{i=1}^n\C^N(b_i)$ of
the Yangian $\Yn$ evaluation vector representations. We consider the action of
the commutative Bethe subalgebra $\Bcq\subset\Yn$ on a $\gln$-weight subspace
$V(\bb)_\bla\subset V(\bb)$ of weight $\bla$. Here the Bethe algebra
depends on the parameters $\qq=(q_1\lc q_N)$. We identify the
$\Bcq$-module $V(\bb)_\bla$ with the regular representation of the algebra
of functions on a fiber of a suitable discrete Wronski map.
If $\qq=(1\lc 1)$, we study the action of $\Bo$ on a space
$V(\bb)^\sing_\bla$ of singular vectors of a certain weight. Again, we
identify the $\Bo$-module $V(\bb)^\sing_\bla$ with the regular
representation of the algebra of functions on a fiber of another suitable
discrete Wronski map.

These results we announced earlier in relation with a description of
the quantum equivariant cohomology of the cotangent bundle of a partial flag
variety and a description of commutative subalgebras of the group algebra of
a symmetric group.
\end{abstract}

\section{Introduction}
A Bethe algebra of a quantum integrable model is a commutative algebra of
linear operators (Hamiltonians) acting on the vector space of states of the model.
An interesting problem is to describe the Bethe algebra as the algebra of
functions on a suitable scheme. Such a description can be considered as
an instance of the geometric Langlands correspondence, see for example \cite{MTV4}.
The $\gln$ \XXX/ model is an example of a quantum integrable model. The Bethe
algebra \,$\Bck\<$ of the \XXX/ model is a commutative subalgebra of the
Yangian $\Yn$. The algebra \,$\Bck\<$ depends on the parameters
$\kk=(\kkk)\in\C^N$. Having a $Y(\gln)$-module $M$, one obtains
the commutative subalgebra $\Bck(M)\subset\End(M)$ as the image of \,$\Bck\<$.
The geometric interpretation of the algebra $\Bck(M)$ as the algebra of
functions on a scheme leads to interesting objects, see for example,
\cite{GRTV}.

In this paper, we consider (among other Yangian modules) a tensor product $V(\bb)= \otimes_{i=1}^n\C^N(b_i)$ of
the Yangian $\Yn$ evaluation vector representations. We consider the action of
the Bethe subalgebra $\Bcq\subset\Yn$ on a $\gln$-weight subspace
$V(\bb)_\bla\subset V(\bb)$ of weight $\bla$.
We identify the
$\Bcq$-module $V(\bb)_\bla$ with the regular representation of the algebra
of functions on a fiber of a suitable discrete Wronski map. If
$\qq=(1\lc 1)$, we study the action of $\Bo$ on a space
$V(\bb)^\sing_\bla$ of singular vectors of a certain weight. Again, we
identify the $\Bo$-module $V(\bb)^\sing_\bla$ with the regular
representation of the algebra of functions on a fiber of another suitable
discrete Wronski map.

\vsk.2>
These results are parallel to the analogous results of \cite{MTV3, MTV4},
where we study the corresponding $\gln[t]$-modules instead of the Yangian
$\Yn$-modules.

\vsk.2>
We used the results of this paper earlier in \cite[Theorems 6.3-6.5]{GRTV}
in relation with a description of the quantum equivariant cohomology of
the cotangent bundle of a partial flag variety and in \cite[Theorem 7.3]{MTV6}
in relation with a description of commutative subalgebras of the group algebra
of a symmetric group. More details are given in remarks after
Theorem \ref{1 thm} and at the end of Section \ref{spoly}.

\vsk.2>
In Section \ref{alg sec}, we consider the space
$\V=(\C^N)^{\otimes n}\otimes\C[\zzz]$, an action on $\V$ of the
symmetric group $S_n$, the subspace $\V^S\subset \V$ of the $S_n$-invariants,
the \,$\gln\!$ weight subspaces $(\V^S)_\bla\subset\V^S$ and the subspaces
$(\V^S)^\sing_\bla\subset(\V^S)_\bla$ of singular vectors.

\vsk.2>
In Section \ref{Yang modules}, we introduce an action of the Yangian $\Yn$
on $\V^S$. In Section \ref{sec Bethe algebra}, we introduce Bethe subalgebras
$\Bcq\subset\Yn$. The induced $\Bcq$-action on $\V^S$ preserves the weight
subspaces $(\V^S)_\bla$. If $\qq=(1\lc 1)$, then the $\Bo$-action
on $\V^S$ preserves the subspaces $(\V^S)^\sing_\bla$ of singular vectors.

\vsk.2>
In Section \ref{sqe}, we introduce a discrete Wronski map on collections
of quasi-exponentials. Theorem \ref{1 thm} describes the \,$\Bcq$-module
$(\V^S)_\bla$ for \,$\qq$ with distinct coordinates in terms of the discrete
Wronski map. In Section \ref{spoly} we define a discrete Wronski map on
collections of polynomials. Theorem \ref{2 thm} describes the \,$\Bo$-module
$(\V^S)^\sing_\bla$ in terms of the second Wronski map. Corollaries \ref{1 cor}
and \ref{2 cor} give an application of Theorems \ref{1 thm} and \ref{2 thm}
to a description of the Bethe algebra action on a tensor product of evaluation
vector representations.

\vsk.2>
Proofs of the theorems are based of the Bethe ansatz. We prove the
corresponding Bethe ansatz statements in Sections \ref{Sec BA q ne 1}
and \ref{Sec BA q 1}, and prove Theorems \ref{1 thm} and \ref{2 thm}
in Section \ref{sec proofs}.

\vsk.2>
In Section \ref{alg sec A}, we consider the $S_n$-skew-invariant part
\>$\V^A\subset \V$ and the space \,$\DVe$ of $S_n$-invariant rational
functions. Theorems \ref{3 thm} and \ref{4 thm} describe the \,$\Bcq$-module
$\DLe$ \,for \,$\qq$ with distinct coordinates and the \,$\Bo$-module \,$\DLs$
in terms of the corresponding Wronski maps.


\section{Space \,$\V^S$}
\label{alg sec}

\subsection{Lie algebra $\gln$}
\label{sec gln}

Let $e_{\ij}$, $i,j=1\lc N$, be the standard generators of the Lie algebra
$\gln$ satisfying the relations
$[e_{\ij},e_{\kl}]=\dl_{\jk}e_{\il}-\dl_{\il}e_{\kj}$. We denote by
$\h\subset\gln$ the subalgebra generated by $e_{\ii},\,i=1\lc N$. For a Lie
algebra $\g\,$, we denote by $U(\g)$ the universal enveloping algebra of $\g$.

\vsk.2>
A vector $v$ of a $\gln$-module $M$ has weight
$\bla=(\la_1\lc\la_N)\in\C^N$ if $e_{\ii}\>v=\la_i\>v$ for $i=1\lc N$.
A vector $v$ is singular if $e_{\ij}v=0$ for $1\le i<j\le N$.

\vsk.2>
We denote by $M_{\bs\la}$ the subspace of $M$ of weight $\bs\la$,
by $M^\sing$ the subspace of $M$ of all singular vectors and by
$M^\sing_{\bs\la}$ the subspace of $M$ of all singular vectors
of weight $\bs\la$.

\vsk.2>
A sequence of integers $\bs\la=(\la_1\lc\la_N)$ such that
$\la_1\ge\la_2\ge\dots\ge\la_N\ge0$ is called a partition with at most $N$
parts. Set $|\bs\la|=\sum_{i=1}^N\la_i$. We say that $\bs\la$ is a partition
of $|\bs\la|$.

\vsk.2>
Let $\C^N$ be the standard vector representation of $\gln$ with basis
$v_1\lc v_N$ such that $e_{\ij}v_k=\dl_{\jk}v_i$ for all $i,j,k$.
A tensor power $V=(\C^N)^{\ox n}$ of the vector representation has a basis
given by the vectors $v_{i_1}\!\lox v_{i_n}$, where
\;$i_j\in\{1\lc N\}$. Every such sequence $(i_1\lc i_n)$ defines
a decomposition $I=(I_1\lc I_N)$ of $\{1\lc n\}$ into disjoint subsets
$I_1\lc I_N$, \,where \;$I_j=\{k\ |\ i_k=j\}$. We denote the basis vector
$v_{i_1}\!\lox v_{i_n}$ by $v_I$.

Let
\vvn-.2>
\be
V\,=\!\bigoplus_{\bla\in\Z^N_{\geq 0},\,|\bla|=n}\!V_\bla
\ee
be the weight decomposition. Denote $\Il$ the set of all indices $I$ with
$|I_j|=\la_j$, \;$j=1,\dots N$. The vectors $\{v_I\ |\ I\in\Il\}$ form a basis of
$V_\bla$.

\subsection{Space $\V^S$}
\label{sec S-actions}

Let $\V$ be the space of polynomials in $\zz=(\zzz)$ with coefficients
in $V=(\C^N)^{\ox n}$:
\vvn-.3>
\be
\V\>=\,V\<\ox_{\C}\C[\zzz]\,.
\vv.3>
\ee
We embed the space \,$V$ into $\V$ \>by sending
\,$v\in V$ to \,$v\otimes 1\in\V$.

Consider the grading on $\C[\zzz]$, $\deg z_i=1$ for
$i=1\lc n$. We define the degree of elements of $\V$ by the rule
$\deg(v\ox p)=\deg p$. We consider the increasing filtration
$F_0\V\subset F_1\V \subset \dots \subset \V$
whose $k$-th subspace consists of elements of degree $\leq k$.
The filtration on $\V$ induces a natural filtration on $\End(\V)$.

Let $P^{(\ij)}$ be the permutation of the $i$-th and $j$-th factors
of $V=(\C^N)^{\ox n}$. Let $s_1\lc s_{n-1}\in S_n$ be the elementary transpositions.
We define an $S_n$-action on \,$V$-valued functions of \>$\zzz$
by the formula:
\begin{align}
\label{Sn+}
s_i: f(\zzz)\,\mapsto\,
\frac{(z_i-z_{i+1})\,P^{(\ii+1)}-1}{z_i-z_{i+1}}\;
& f(\zzii)\,+{}
\\[3pt]
{}+\,\frac{1}{z_i-z_{i+1}}\;&f(\zzi)\,.
\notag
\\[-14pt]
\notag
\end{align}
These formulae induce an \,$S_n$-action on $\V$. The $S_n$-action preserves
the filtration: for any $k$ we have $S_n\times F_k\V \to F_k\V$. We denote
by $\V^S$ the subspace of $S_n$-invariants in $\V$.

\vsk.2>
The group $S_n$ acts on the algebra $\C[\zzz]$ by permuting
the variables. Let $\si_i(\zb)$, $i=1,\alb\dots,n$, be the $s$-th elementary
symmetric polynomial in $\zzz$. The algebra of symmetric polynomials
$\Czs$ is a free polynomial algebra with generators
$\si_1(\zb)\lc\alb\si_n(\zb)$.

\begin{lem}
\label{VSfree}
The space $\V^S$ is a free $\Czs$-module of rank $N^n$.
\end{lem}

\begin{proof} The lemma follows from Lemma 2.10 in \cite{GRTV}.
\end{proof}

The Lie algebra $\gln$ naturally acts on $\V$ preserving the grading and
commuting with the $S_n$-action on $\V$. Therefore, \,$\V^S\<$ is a filtered
$\gln$-module. We consider the $\gln$-weight decomposition
\vvn-.1>
\be
\V^S=\!\tbigoplus_{\fratop{\bla\in\Z^N_{\geq 0}}{|\bla|=n}}\!(\V^S)_\bla\,,
\vv.2>
\ee
as well as the subspaces of singular vectors $(\V^S)^\sing_\bla\!\subset
(\V^S)_\bla$. All of these are filtered free $\Czs$-modules.

\vsk.2>
Let $M$ be a $\Z_{\ge0}$-filtered space with finite-dimensional graded
components $F_kM/F_{k-1}M$.
We call the formal power series in a variable $t$,
\vvn.2>
\be
\on{ch}_M(t)\,=\,\sum_{k=0}^\infty\,(\dim F_kM/F_{k-1}M)\,t^k\,,
\vv.2>
\ee
the graded character of $M$. We set $\,(\<\>t\<\>)_a=\prod_{j=1}^a(1-t^j)\,$.

\begin{lem}
\label{lem on char of VSl}
For $\bs\la\in\Z^N_{\geq 0}$, $|\bla|=n$, we have
\vvn-.2>
\beq
\label{forMula}
\on{ch}_{\>(\V^S)_{\bs\la}}(t)\,=\,
{\prod_{i=1}^N\frac1{(\<\>t\<\>)_{\la_i}}}\;.
\eeq
For a partition $\bla$ of $n$ with at most $N$ parts, we have
\vvn.2>
\beq
\label{forMulaSing}
\on{ch}_{\>(\V^S)_{\bs\la}^\sing}(t)\,=\,
\frac{\prod_{1\le i<j\le N}\,(1-t^{\la_i-\la_j+j-i})}
{\prod_{i=1}^N\,(\<\>t\<\>)_{\la_i+N-i}}\ t^{\,\sum_{i=1}^N{(i-1)\la_i}}.
\eeq
\end{lem}

\begin{proof}
The $S_n$-action on the graded components $F_k\V_{\bs\la}/F_{k-1}\V_{\bs\la}$
and $F_k\V_{\bs\la}^\sing/F_{k-1}\V_{\bs\la}^\sing$ coincides with
the $S_n$-action considered in \cite{MTV2} and \cite{MTV3}.
Formula \Ref{forMula} follows from \cite[Lemma 2.12]{MTV3}.
Formula \Ref{forMulaSing} follows from \cite[Formula 5.3]{MTV2} and
\cite[Lemma 2.2]{MTV2}.
\end{proof}

Given $\aa=(a_1\lc a_n)\in\C^n$, denote by
$I_\aa\subset\Czs$ the ideal generated by the polynomials
$\si_i(\zb)-a_i$, \,$i=1\lc n$. For any $\aa$, the quotient
$\V^S\<\</I_\aa\V^S$ is a complex vector space of dimension $N^n$
by Lemma \ref{VSfree}.

\vsk.2>
\section{Yangian modules}
\label{Yang modules}

\subsection{Yangian $\Yn$}
\label{sec yangian}

The Yangian $\Yn$ is the unital associative algebra with generators
\,$T_{\ij}\+s$ \,for \,$i,j=1\lc N$, \;$s\in\Z_{>0}$, \,subject to relations
\vvn.3>
\beq
\label{ijkl}
(u-v)\>\bigl[\<\>T_{\ij}(u)\>,T_{\kl}(v)\<\>\bigr]\>=\,
T_{\kj}(v)\>T_{\il}(u)-T_{\kj}(u)\>T_{\il}(v)\,,\qquad i,j,k,l=1\lc N\,,
\vv.3>
\eeq
where
\vvn-.7>
\be
T_{\ij}(u)=\dl_{\ij}+\sum_{s=1}^\infty\,T_{\ij}\+s\>u^{-s}\>.
\vv.2>
\ee
The Yangian $\Yn$ is a Hopf algebra with the coproduct
\;$\Dl:\Yn\to\Yn\ox\Yn$ \,given by
\vvn-.4>
\be
\Dl\,:\,T_{\ij}(u)\,\mapsto\,\sum_{k=1}^N\,T_{\kj}(u)\ox T_{\ik}(u)
\vv.2>
\ee
\,for \,$i,j=1\lc N$\>. The Yangian $\Yn$ contains \>$\Ugln$ \,as a Hopf
subalgebra, the embedding given by \,$e_{\ij}\mapsto T_{\ji}\+1$.

The Yangian $\Yn$ has the degree function such that
$\on{deg}T_{i,j}^{\{s\}}=\>s-1$ for any \,$i,j=1\lc N$, \,$s=1,2,\ldots$.
The Yangian $\Yn$ is a filtered algebra with the increasing filtration
$F_0\Yn\subset F_1\Yn \subset\dots\subset\Yn$, where $F_s\Yn$ consists of
elements of degree $\leq s$.

\vsk.2>
There is a one-parameter family of automorphisms
\vvn.3>
\be
\rho_b:\>\Yn\to\Yn,\qquad T_{i,j}(u) \mapsto T_{i,j}(u-b)\,,
\vv.3>
\ee
where $b\in\C$ and $(u-b)^{-1}$ in the right-hand side has to be expanded
as a power series in $u^{-1}$.

The evaluation homomorphism \;$\epe:\Yn\to U(\glN)$ is defined by the rule:
$T^{\{1\}}_{i,j} \mapsto e_{j,i}$ for all $i,j$, and $T^{\{s\}}_{i,j}\mapsto 0$
for all $i,j$ and all $s>1$.

For a $\glN$-module $M$ and $b\in\C$, we denote by $M(b)$ the $\Yn$-module
induced from $M$ by the homomorphism $\epe\cdot\rho_b$. We call it the
evaluation module with the evaluation point $b$.

Recall that we consider $\C^N$ as the $\glN$-module with highest weight
$(1,0\lc 0)$. For any $\bb=(b_1\lc b_n)\in \C^n$, we obtain the $\Yn$-module
\vvn.3>
\be
V(\bb)\,=\,\C^N(b_1)\lox \C^N(b_n)\,.
\vv.3>
\ee

\subsection{Yangian module $\V^S$}
Consider $\C^N\!\ox V=\C^N\ox (\C^N)^{\ox n}$, where the factors are labeled by
\,$0,1\lc n$. Set
\vvn-.1>
\be
\label{Lpm}
L(u)\,=\,(u-z_n+\<\>P^{(0,n)})\dots(u-z_1+\<\>P^{(0,1)})\,.
\vv.3>
\ee
This is a polynomial in \,$u,\>\zzz$ \,with values
in \,$\End(\C^N\!\ox V)$. We consider $L(u)$ as an $N\!\times\!N$ matrix
with \,$\End(V)\ox\C[u,\zzz]\>$-valued entries \,$L_{\ij}(u)$, \,$i,j=1\lc N$.

\begin{lem}
The assignment
\vvn-.3>
\beq
\label{pho}
\pho\ :\ T_{\ij}(u)\ \mapsto \ L_{\ij}(u)\,\prod_{a=1}^n\,(u-z_a)^{-1}
\eeq
defines the $\Yn$-module structure on $\V = (\C^N)^{\ox n}\ox \C[\zzz]$. We consider the right-hand side of \Ref{pho} as a series
in \,$u^{-1}$ with coefficients in \,$\End(V)\ox\Czh$\>.
\end{lem}
\begin{proof}
The Yang-Baxter equation
\vvn.3>
\begin{align}
\label{YB}
(u-v+h\<\>P^{(1,2)})\> &(u+h\<\>P^{(1,3)})\>(v+h\<\>P^{(2,3)})\,=\,
\\[3pt]
{}=\,{}&(v+h\<\>P^{(2,3)})\>(u+h\<\>P^{(1,3)})\>(u-v+h\<\>P^{(1,2)})\,.
\notag
\\[-22pt]
\notag
\end{align}
implies that
\vvn-.1>
\be
(u-v+P^{(1,2)})\,L^{(1)}(u)\>L^{(2)}(v)\,=\,
L^{(2)}(v)\>L^{(1)}(u)\,(u-v+P^{(1,2)})\,,
\vv.1>
\ee
which means
\vvn-.3>
\be
(u-v)\>\bigl[\<\>L_{\ij}(u)\>,L_{\kl}(v)\<\>\bigr]\>=\,
L_{\kj}(v)\>L_{\il}(u)-L_{\kj}(u)\>L_{\il}(v)
\vv.3>
\ee
for all \,$i,j,k,l=1\lc N$. Comparing the last formula with the defining
relations \Ref{ijkl} for the Yangian \>$\Yn$ completes the proof.
\end{proof}

The subalgebra $U(\gln)\subset\Yn$ acts on $\V$ in the standard way:
an element \,$x\in\gln$ \,acts as $x^{(1)}\lsym+x^{(n)}$.

\begin{lem}
\label{lem filt}
The $\Yn$-action on $\V$ is filtered: for any $k,s$, we have $F_s\Yn\times F_k\V \to F_{s+k}\V$.
\qed
\end{lem}

\begin{lem}
\label{pm & Yan 1}
The $\Yn$-action \,$\pho$ on $\V$ commutes with the $S_n$-action \Ref{Sn+}
and with multiplication by the elements of \,$\C[z_1,\dots,z_n]$.
\end{lem}
\begin{proof}
The first part follows from the Yang-Baxter equation \Ref{YB}, and
the second part is clear.
\end{proof}

By Lemma \ref{pm & Yan 1}, the action \,$\pho$ makes the space \,$\V^S$ into
a filtered $\Yn$-module. For any $\aa=(a_1\lc a_n)\in\C^N$, the subspace
$I_\aa\V^S$ is a $\Yn$-submodule.

\begin{lem}[{\cite[Proposition 4.6]{GRTV}}]
\label{V+v1 1}
The \,$\Yn$-module \;$\V^S\!$ is generated by the vector
\,$\voxn\<\<=\<\>\vone$.
\qed
\end{lem}

For $\aa=(a_1\lc a_n)\in\C^N$, introduce complex numbers \,$b_1\lc b_n$
by the relation
\beq
\label{ab}
\prod_{s=1}^n\,(u-b_s)\,=\,u^n+\>\sum_{j=1}^n (-1)^j\>a_j\>u^{n-j}.
\eeq
The numbers are defined up to a permutation.

\begin{prop}
\label{VSVb}
Assume that the numbers \,$b_1\lc b_n$ are ordered such that \,$b_i\ne b_j+1$
for $i>j$. Then the $\Yn$-module $\V^S\<\</I_\aa\V^S$ is isomorphic to
$V(\bb)=\C^N(b_1)\lox\C^N(b_n)$, the tensor product of
evaluation $\Yn$-modules.
\end{prop}
\begin{proof}
Consider the map \,$\phi:\V^S\!\to V(\bb)$ that sends every element of \,$\V^S$
to its value at the point \,$\zb=(b_1\lc b_n)$\>. This map is
a homomorphism of $\Yn$-modules and factors through the canonical projection
\,$\thi:\V^S\!\to\V^S\<\</I_\aa\V^S$. Since \,$\thi$ \>is also a homomorphism
of $\Yn$-modules, this defines a homomorphism of $\Yn$-modules
\,$\psi:\V^S\<\</I_\aa\V^S\!\to V(\bb)$.

\vsk.2>
Under the assumption that \,$b_i\ne b_j+1$ for \,$i>j$\>, the $\Yn$-module
\,$V(\bb)$ is generated by the vector \,$\voxn$\>, see
\cite[Proposition~3.1]{NT2}. Therefore, the map \,$\psi$
\,is surjective because \,$\psi(\voxn)=\voxn$\>, and since
\;$\dim\<\>\V^S\<\</I_\aa\V^S\<=N^n\<=\dim\<\>V(\bb)$,
the map \,$\psi$ \,is an isomorphism of $\Yn$-modules.
\end{proof}

\begin{prop}
\label{irr}
The $\Yn$-module $V(\bb)$ is irreducible if and only if
\,$b_i\ne b_j+1$ for all \,$i\ne j$.
\end{prop}
\begin{proof}
The statement follows, for instance, from \cite[Theorem~3.4]{NT2}.
\end{proof}

\section{Bethe subalgebras}
\label{sec Bethe algebra}

\subsection{Bethe subalgebras}
\label{sec Bethe alg}
For \;$k=1\lc N$, \;$\ib=\{1\leq i_1<\dots<i_k\leq N\}$,
\;$\jb=\{1\leq j_1<\dots<j_k\leq N\}$, define
\vvn-.3>
\be
M_{\ijb}(u) =\sum_{\si\in S_k}(-1)^\si\,
T_{i_1,j_{\si(1)}}(u)\dots T_{i_k,j_{\si(k)}}(u-k+1)\,.
\ee
For \,$\ib=\{1\lc N\}$, \,the series \,$M_{\iib}(u)$ is called the quantum
determinant and denoted by $\qdet T(u)$. Its coefficients generate the center
of the Yangian \,$\Yn$.

\goodbreak
\vsk.2>
For \,$\qq = (\kkk)\in(\C^*)^N$ and \,$k=1\lc N$, \, we define
\beq
\label{Bp}
\Bq_k(u)\,=\!\sum_{\ib\>=\<\>\{1\leq i_1<\<\dots<\>i_k\leq N\}}\!\!
\kk_{i_1}\!\dots\kk_{i_k}\>M_{\iib}(u)\,=\,\si_k(\kkk)+
\sum_{s=1}^\infty\>\Bq_{k,s}\>u^{-s}\>,\kern-1em
\vv.2>
\eeq
where \,$\si_k$ is the \,$k$-th elementary symmetric function and
\,$\Bq_{k,s}\in\Yn$. In particular,
\vvn.2>
\be
\Bq_N(u)\,=\,q_1\dots q_N\>M_{\iib}(u),
\vv.2>
\ee
where $\ib=\{1\lc N\}$. The generating series $\Bq_k(u)$, \,$k=1\lc N$, are
called the transfer-matrices.

\begin{lem}
We have $B_{k,s}^\qq\in F_s\Yn$ for all $k,s$.
\qed
\end{lem}

Let \,$\Bcq\!\subset\Yn$ \,be the unital subalgebra generated by the elements
\,$\Bq_{k,s}$\>, \,$k=1\lc N$, \,$s>0$. The subalgebra \,$\Bcq\<$ is called
a Bethe subalgebra of $\Yn$. The subalgebra \,$\Bcq\<$ does not change if
all $\kkk$ are multiplied by the same number. If $\qq=(1\lc 1)$, then
the corresponding Bethe subalgebra will be denoted by $\Bo$.

\begin{thm}[\cite{KS}]
\label{BY}
The subalgebra \,$\Bcq$ is commutative and commutes with the subalgebra
$U(\h)\subset \Yn$. The subalgebra \,$\Bo$ commutes with the subalgebra
$U(\glN)\subset \Yn$.
\qed
\end{thm}

\medskip
As a subalgebra of $\Yn$, the Bethe algebra \,$\Bcq$ acts on any
$\Yn$-module $M$. Since \,$\Bcq$ commutes with $U(\h)$, it preserves the weight
subspaces $M_\bla$. The subalgebra $\Bo$ preserves the singular weight
subspaces $M_\bla^\sing$.

\vsk.2>
If $L\subset M$ is a \,$\Bcq$-invariant subspace,
then the image of \,$\Bcq$ in $\End(L)$ will be called the Bethe algebra of $L$
and denoted by \,$\Bcq(L)$.

\vsk.2>
We will study the action of $\Bcq$ on the weight subspaces $(\V^S)_\bla$ and
the action of $\Bo$ on the singular weight subspaces $(\V^S)_\bla^\sing$. The
image of $\Bq_{k,s}$ in $\End((\V^S)_\bla)$ will be denoted by $\Bql_{k,s}$.

\begin{lem}
The generating series $\Bq_N(u)$ acts on the $\Yn$-module $\V$ as
multiplication by the scalar function
\vvn-.3>
\beq
\label{BN}
q_1\dots q_N\,\prod_{i=1}^n\frac {u-z_i+1}{u-z_i} \,.
\vv->
\eeq
\qed

\end{lem}

\begin{cor}
The Bethe algebra \,$\Bcq(\V)$ contains the algebra of scalar operators of
multiplication by elements of \;$\Czs$.
\qed

\end{cor}

\subsection{Universal difference operator}
\label{univ diff oper}

Define the operator $\tau$ acting on functions of $u$ as $(\tau f)(u)=f(u-1)$.
Following \cite{T}, for $\qq=(q_1\lc q_N)\in(\C^\times)^N$ we introduce
the universal difference operator $\D^\qq(u,\tau)$ by the formula
\vvn-.3>
\beq
\label{UDO}
\D^\qq(u,\tau)\,=\,1 + \sum_{k=1}^N\,(-1)^kB_k^\qq(u)\,\tau^k\>.
\vv-.4>
\eeq
For $\qq=\bs 1$, we write
\vvn-.2>
\beq
\label{formula for Se}
\D^{\qq=\bs 1}(u,\tau)\tau^{-N}=\,
\sum_{k=0}^N \, (-1)^k\, C_k(u) \,(\tau^{-1} - 1)^{N-k},\qquad
C_k(u) = \sum_{s=0}^\infty C_{k,s} u^{-s},
\eeq
where $C_{k,s}\in\Bo$. The Bethe algebra $\Bo$ preserves $(\V^S)_\bla^\sing$
and we may consider the images $\Cl_{k,s}$ of the elements $C_{k,s}$
in $\Bo\bigl((\V^S)^\sing_\bla\bigr)$.

\begin{thm}[{\cite[Theorem 3.7]{MTV2}}]
\label{thm Q=1}
The following statements hold.
\begin{enumerate}\itemsep=4pt
\item[(i)]
$C_0(u) = 1$\>.
\item[(ii)]
$C_{k,s}=0$ \,for all \,$k=1\lc N$ and \,$s<k$\>.
\item[(iii)]
$\Cl_{1,1}\lc\Cl_{N,N}$ are scalar operators, and for a variable $x$, we have
\be
\sum_{k=0}^N\,(-1)^k\,\Cl_{k,k}\prod_{j=0}^{N-k-1}\!(x-j)\,=\,
\prod_{s=1}^N\,(x-\la_s-N+s)\,.
\vv-1.3>
\ee
\qed
\end{enumerate}
\end{thm}

\begin{rem}
Given an $N\times N$ matrix $A$ with possibly noncommuting entries $a_{i,j}$,
we define its row determinant to be
\vvn-.2>
\be
\rdet A\,=
\sum_{\;\si\in S_N\!} (-1)^\si\,a_{1,\si(1)}a_{2,\si(2)}\dots a_{N,\si(N)}\,.
\vv-.1>
\ee
The universal difference operator can be presented as a row determinant
of a suitable matrix, see for example \cite{T, MTV1, MTV2}.
\end{rem}

\section{Spaces of quasi-exponentials}
\label{sqe}

\subsection{Spaces of quasi-exponentials}
\label{Spaces of quasi-exponentials}

Let $\qq=(q_1\lc q_N)\in(\C^\times)^N$ be
\vv.1>
a sequence of distinct numbers. Let $\bla \in\Z^N_{\geq 0}$, $|\bla|=n$.
\vv.1>
Let $\Omql$ be the affine $n$-dimensional space with coordinates
$p_{i,j},\,i=1\lc N,\,j=1\lc\la_i$.

Introduce \,$f_i(u)= q_i^u p_i(u)$, $i=1\lc N$, where
\vvn.3>
\beq
\label{basis}
p_i(u)\,=\,u^{\la_i} + p_{i,1}u^{\la_i-1} + \dots +p_{i,\la_i}\,.
\vv.3>
\eeq
We identify points \>$X\in\Omql$ with $N$-dimensional complex
vector spaces generated by quasi-exponentials
\vvn.2>
\beq
\label{basis X}
f_i(u,X)\,=\,q_i^u(u^{\la_i} + p_{i,1}(X)u^{\la_i-1} + \dots
+p_{i,\la_i}(X))\,, \qquad i=1\lc N.
\eeq

\vsk.2>
Denote by $\Oql$ the algebra of regular functions on
$\Omql$. It is the polynomial algebra in the variables $p_{i,j}$.
The algebra $\Oql$ has the degree function such that
$\deg p_{i,j}=j$ for all $i,j$. We consider the
the increasing filtration $F_0\Oql \subset F_1\Oql \subset\dots\subset \Oql$, where
$F_s\Oql$ consists of elements of degree $\leq s$.
The graded character of $\Oql$ is
\be
\on{ch}_{\>\Oql}(t)\,=\,{\prod_{i=1}^N\frac1{(\<\>t\<\>)_{\la_i}}}\;.
\ee

\subsection{Another realization of $\Oql$}
For arbitrary functions $g_1(u)\lc g_N(u)$, we introduce the discrete Wronskian
by the formula
\vvn-.1>
\be
\Wr\>(g_1(u)\lc g_N(u))\,=\,
\det\left(\begin{matrix} g_1(u) & g_1(u-1) &\dots & g_1(u-N+1) \\
g_2(u) & g_2(u-1) &\dots & g_2(u-N+1)
\\ \dots & \dots &\dots & \dots\\
g_N(u) & g_N(u-1) &\dots & g_N(u-N+1)
\end{matrix}\right).
\ee
Let $f_i(u)$, $i=1\lc N$, be the functions given by \Ref{basis}. We have
\vvn.2>
\beq
\label{Wr coef}
\Wr\>(f_1(u-1)\lc f_N(u-1))\,=\,
\prod_{i=1}^N q_i^{u-1}\!\prod_{1\le i<j\le N}(q^{-1}_j-q^{-1}_i)
\ \Bigl(u^n+\sum_{s=1}^n \>(-1)^s\Sig_s\,u^{n-s}\Bigr)\,,\kern-1em
\eeq
where $\Sig_1\lc\Sig_n$ are elements of $\Oql$.
Define the difference operator $\D^{\Oql}(u,\tau)$ by
\vvn.2>
\beq
\label{DOla}
\D^{\Oql}(u,\tau)\,=\,\frac{1}{\Wr\>(f_1(u-1)\lc f_N(u-1))}\,\rdet
\left(\begin{matrix} f_1(u) & f_1(u-1) &\dots & f_1(u-N) \\
f_2(u) & f_2(u-1) &\dots & f_2(u-N) \\ \dots & \dots &\dots & \dots \\
1 & \tau &\dots & \tau^N
\end{matrix}\right).\kern-1em
\eeq
It is a difference operator in the variable $u$, whose coefficients
are formal power series in $u^{-1}$ with coefficients in $\Oql$,
\beq
\label{DO}
\D^{\Oql}(u,\tau)\,=\, 1 +\sum_{k=1}^N\,(-1)^k\Fql_k(u)\,\tau^k\>,
\qquad
\Fql_k(u)\,=\,\si_k(q_1\lc q_N) +
\sum_{s=1}^\infty \Fql_{k,s}\>u^{-s}\,,
\eeq
and $\Fql_{k,s}\in\Ol$\>, \,$k=1\lc N$, \,$s>0$\>.
In particular, we have
\beq
\label{FN}
\Fql_N(u)\,=\,q_1\dots q_N \,
\frac{(u+1)^n+\sum_{s=1}^n \>(-1)^s\Sig_s\,(u+1)^{n-s}}
{u^n+\sum_{s=1}^n \>(-1)^s\Sig_s\,u^{n-s}},
\eeq
cf.~\Ref{BN}.

\begin{lem}
\label{coef alg}
The functions $F_{k,s}^{\qq,\bla}\in\Oql$, \,$k=1\lc N$, $s>0$\>,
generate the algebra \,$\Oql$.
\end{lem}
\begin{proof}
The coefficient of $u^{\la_i-j-1}$ of the series
$q_i^{u}\,\D^{\Oql}f_i(u)$ has the form
\vvn.3>
\be
\label{fij}
-\>j\,p_{i,j}\prod_{\satop{k=1}{k\ne i}}^N\,(1-q_k/q_i)\,+\,
\sum_{l=1}^N\,\sum_{r=0}^{j+1}\,\sum_{s=0}^{j-1}\,
c_{ijlrs}\,\Fql_{l,r}\,p_{i,s}\,,
\vv.1>
\ee
where \,$c_{ijlrs}$ are some numbers, $\Fql_{l,0}=\si_l(q_1\lc q_N)$ and
$p_{i,0}=1$. Since $\D^{\Oql}f_i(u)=0$, we can express recursively the elements
$p_{i,j}$ via the elements $\Fql_{l,r}$ starting with $j=1$ and then increasing
the second index $j$.
\end{proof}

\subsection{Discrete Wronski map $\pi^\qq_\bla$}
\label{wronski}
Consider \,$\C^n$ with coordinates $\si_1\lc\si_n$.
Introduce the {discrete Wronski map\/} $\pi^\qq_\bla:\Omql\to\C^n$ as follows.
Let $X$ be a point of $\Omql$. Define
\vvn.3>
\beq
\label{wronsk of X}
\Wr_X(u)\,=\,\Wr\>\bigl(f_1(u-1,X)\lc f_N(u-1,X)\bigr)\,,
\vv.3>
\eeq
where $f_1(u,X)\lc f_N(u,X)$ are given by \Ref{basis X}. Let
\be
\Wr_X(u)\,=\prod_{i=1}^N q_i^{u-1}\!\prod_{1\le i<j\le N}(q^{-1}_j-q^{-1}_i)
\ \Bigl(u^n+\sum_{s=1}^n \>(-1)^sa_s\,u^{n-s}\Bigr)\,.
\ee
We set $\pi^\qq_\bla: X\mapsto (a_1\lc a_n)$.

\vsk.2>
The discrete Wronski map is a finite algebraic map, see
\cite[Proposition 3.1]{MTV5}. It defines an injective algebra homomorphism
\vvn.2>
\be
(\pi^\qq_\bla)^* :\,\C[\si_1\dots,\si_n] \to \O^\qq_\bla ,\qquad
\si_s \mapsto \Si_s\,,
\ee
which gives a $\C[\si_1\dots,\si_n]$-module structure on $\O^\qq_\bla$.

\vsk.2>
For $\aa\in\C^n$, let $I^\Oc_\lba$ be the ideal in $\Oql\>$ generated
by the elements $\Sig_s-a_s$, $s=1\lc n$, where $\Sig_1\lc\Sig_n$ are defined
by \Ref{Wr coef}. The quotient algebra
\vvn.3>
\beq
\label{Olaa}
\Oq_\lba\>=\,\Oql\>\big/I^\O_\lba
\vv.3>
\eeq
is the scheme-theoretic fiber of the discrete Wronski map $\pi^\qq_\bla$.

\subsection{First main result}
Let \,$\Ac$ \,be a commutative algebra. The algebra \,$\Ac$ \,considered as
a module over itself is called the regular representation of \,$\Ac$\>.

\begin{thm}
\label{1 thm}
Assume that $\qq\in(\C^\times)^N$ has distinct coordinates.
Denote $v_\bla = \sum_{I\in\Il}v_I$. Then
\begin{enumerate}
\item[(i)]
The map $\mu^\qq_\bla\!:\Fql_{k,s} \mapsto \Bql_{k,s}$,
\,$k=1\lc N$, \,$s>0$, extends uniquely to an isomorphism
$\mu^\qq_\bla\!: \O^\qq_\bla\<\to\Bcq\bigl((\V^S)_\bla\bigr)$
of filtered algebras. The isomorphism $\mu^\qq_\bla$ becomes an isomorphism of
the $\C[\si_1\lc\si_n]$-module \,$\O^\qq_\bla$ and the
$\Czs$-module \,$\Bcq\bigl((\V^S)_\bla\bigr)$ if we identify the algebras
$\C[\si_1\lc\si_n]$ and $\Czs$ by the map
$\si_s\mapsto \si_s(\bs z)$, \,$s=1\lc n$.
\vsk.2>
\item[(ii)]
The map \;$\nu^\qq_\bla:\O^\qq_\bla \to\>(\V^S)_\bla$\,,
\,$f\mapsto\mu^\qq_\bla(f)\,v_\bla$\>,
is an isomorphism of filtered vector spaces identifying the
\;$\Bcq\bigl((\V^S)_\bla\bigr)$-module \;$(\V^S)_\bla$ and the regular
representation of \,$\O^\qq_\bla$.
\end{enumerate}
\end{thm}

The theorem is proved in Section \ref{proof thm main 1}.

\begin{rem}
Theorem \ref{1 thm} was announced in \cite[Theorem 6.3]{GRTV}. To indicate
the correspondence of notation, we point out that formula (6.1) in \cite{GRTV}
is a counterpart of formula \Ref{Wr coef} in this paper with functions
\,$g_i(u)$ in \cite{GRTV} being equal to \,$q_i\>h^{\la_i}\<f_i(-1+u/h)$ here.
Also, formulae (6.3), (6.4) in \cite {GRTV} correspond to formulae \Ref{DOla},
\Ref{DO} in this paper, and the algebra \,$\Hck_\bla$ in \cite{GRTV} is
a counterpart of the algebra \,$\Oql$ here.
\end{rem}

Assume that the complex numbers \,$b_1\lc b_n$ are such that \,$b_i\ne b_j+1$
for $i>j$. Consider the tensor product $V(\bb)=\C^N(b_1)\lox\C^N(b_n)$ of
evaluation $\Yn$-modules and its weight subspace $V(\bb)_\bla$. Introduce
the numbers $\aa=(a_1\lc a_n)$ by the formula $a_s=\si_s(b_1\lc b_n)$,
cf.~\Ref{ab}.

\begin{cor}
\label{1 prime cor}
Assume that $\qq\in({\Bbb R}^\times)^N$ has distinct coordinates.
Let \,$b_1,\dots,b_n$ be real and such that \,$|\>b_i-b_j|>1$ for all $i\ne j$.
Then the algebra \,$\Bcq\bigl(V(\bb)_\bla\bigr)$ has simple spectrum.
\end{cor}

\begin{proof}
Under the assumption made, the algebra $\Bcq\bigl(V(\bb)_\bla\bigr)$
has no nilpotent elements, see \cite[Lemma 3.7 and Lemma 3.10]{MTV5}.
Hence the algebra \,$\Oq_{\bla,\aa}$, and thus the algebra
\,$\Bcq\bigl(V(\bb)_\bla\bigr)$, is the direct sum of one-dimensional
algebras, so the spectrum of $\Bcq\bigl(V(\bb)_\bla\bigr)$ is simple,
see Proposition \ref{VSVb}.
\end{proof}

Other sufficient conditions for simplicity of the spectrum of
$\Bcq\bigl(V(\bb)_\bla\bigr)$ see in \cite[Theorem~2.1, part~(2)]{MTV5}
and in \cite[Theorem~1.1]{MTV7}.

\begin{cor}
\label{1 cor}
Assume that $\qq\in(\C^\times)^N$ has distinct coordinates.
Then the isomorphisms $\mu^\qq_\bla$ and $\nu^\qq_\bla$ induce an isomorphism
of the \;$\Bcq(V(\bb)_\bla)$-module \;$V(\bb)_\bla$ and the regular
representation of the algebra \,$\O^\qq_{\bla,\aa}$.
\qed
\end{cor}

The corollary implies that
\,$\Bcq\bigl(V(\bb)_\bla\bigr)\subset\End\bigl(V(\bb)_\bla\bigr)$
is a maximal commutative subalgebra and \,$\Bcq\bigl(V(\bb)_\bla\bigr)$
is a Frobenius algebra, see for example \cite[Lemma 3.9]{MTV4}.

\section{Spaces of polynomials}
\label{spoly}

\subsection{Spaces of polynomials}
\label{Spaces of polynomials}

Let $\bla \in\Z^N_{\geq 0}$, \,$\la_1\geq \la_2\geq\dots\geq\la_N\geq 0$,
\,$|\bla|=n$. In other words, let $\bla$ be a partition of $n$ with
at most $N$ parts. Introduce the set \,$P=\{d_1>d_2\lsym>d_N\}$\,, where
\,$d_i=\la_i+N-i$. Let \,$\Oml$ be the affine $n$-dimensional space with
coordinates \,$f_{i,j}$, \,$i=1\lc N$, \,$j=1\lc d_i$, \,$d_i-j\notin P$.

\vsk.3>
Introduce polynomials
\vvn-.3>
\beq
\label{basis p}
f_i(u)\,=\,u^{d_i} +
\sum^{d_i}_{\satop{j=1}{d_i-j\,\notin\,P}} f_{i,j}u^{d_i-j},
\qquad i=1\lc N.
\vv.1>
\eeq
We identify points \>$X\in\Oml$ with $N$-dimensional complex
vector spaces generated by polynomials
\vvn-.4>
\beq
\label{basis X p}
f_i(u,X)\,=\,\,u^{d_i} +
\sum^{d_i}_{\satop{j=1}{d_i-j\notin\,P}} f_{i,j}(X)u^{d_i-j}\,,
\qquad i=1\lc N.
\eeq

\vsk.2>
Denote by \,$\Ol$ the algebra of regular functions on \,$\Oml$.
It is the polynomial algebra in the variables $f_{i,j}$. The algebra $\Ol$
has the degree function such that $\deg f_{i,j}=j$ for all $i,j$. We consider
the the increasing filtration $F_0\Ol \subset F_1\Ol \subset\dots\subset \Ol$,
where $F_s\Ol$ consists of elements of degree $\leq s$. The graded character of
$\Ol$ is
\vvn.3>
\be
\ch_{\>\Ol}(t)\,=\,\frac{\prod_{1\leq i<j\leq N}\,(1-t^{\la_i-\la_j +j-i})}
{\prod_{i=1}^N\,(\<\>t\<\>)_{\la_i+N-i}}\;.
\vv-.1>
\ee
see \cite[Lemma 3.1]{MTV4}.

\subsection{Another realization of \,$\Ol$}

Let \,$f_i(u)$, \,$i=1\lc N$, be the generating functions
given by~\Ref{basis p}. We have
\beq
\label{Wr coef q=1}
\Wr\>(f_1(u-1)\lc f_N(u-1))\,=\,\prod_{1\le i<j\le N}(\la_j-\la_i +i-j)
\ \Bigl(u^n+\sum_{s=1}^n (-1)^s\>\Sig_s\,u^{n-s}\Bigr)\,,
\vv.1>
\eeq
where $\Sig_1\lc\Sig_n$ are elements of $\Ol$.
Define the difference operator $\D^{\Ol}$ by
\vvn.3>
\beq
\label{DOla p}
\DOl(u,\tau)\,=\,\frac{1}{\Wr\>(f_1(u-1)\lc f_N(u-1))}\,\rdet
\left(\begin{matrix} f_1(u) & f_1(u-1) &\dots & f_1(u-N) \\
f_2(u) & f_2(u-1) &\dots & f_2(u-N) \\ \dots & \dots &\dots & \dots \\
1 & \tau &\dots & \tau^N
\end{matrix}\right).\kern-1em
\eeq
\beq
\label{formula for G}
\DOl(u,\tau)\tau^{-N}\,=\,
\sum_{k=0}^N \, (-1)^k\, G_k(u)\,(\tau^{-1} - 1)^{N-k}, \qquad
G_k(u)\,=\,\sum_{s=0}^\infty G_{k,s} u^{-s},
\vv-.1>
\eeq
where $G_{k,s} \in \Ol$.

\begin{lem}
\label{thm qq=1}
The following statements hold.
\begin{enumerate}\itemsep=4pt
\item[(i)]
$G_0(u) = 1$\>.
\item[(ii)]
$G_{k,s}=0$ \,for all \,$k=1\lc N$ and \,$s<k$\>.
\item[(iii)]
$G_{1,1}\lc G_{N,N}$ are complex numbers, and for a variable $x$, we have
For all $x$ we have
\be
\sum_{k=0}^N (-1)^k\,G_{k,k}\prod_{j=0}^{N-k-1}\!(x-j)\,=\,
\prod_{s=1}^N\,(x-\la_s-N+s)\,.
\vv-1.3>
\ee
\qed
\end{enumerate}
\end{lem}

\begin{lem}
\label{coef alg p}
The functions \,$G_{k,s}\in\Ol$, $k=1\lc N$, \,$s=k+1,k+2,\dots{}$\>,
generate the algebra \,$\Ol$.
\end{lem}
\begin{proof}
The coefficient of \,$u^{d_i-N-j}$ of the series
\,$\DOl f_i(u)$ has the form $\chi(d_i-j)\,f_{i,j}+\dots{}\;$, where
\,$\chi(x) = \prod_{s=1}^N\,(x - \la_s - N + s)$ and the dots denote the terms
which contain the elements \,$G_{k,l}$ and $f_{i,s}$ with $s<j$ only.
Since \,$\DOl f_i(u)=0$ and \,$\chi(d_i-j)\ne 0$, we can express recursively
the elements \,$f_{i,j}$ via the elements \,$G_{k,l}$ starting with $j=1$ and
then increasing the second index $j$.
\end{proof}

\subsection{Discrete Wronski map $\pi_\bla$}
\label{wronski q=1}
Consider \,$\C^n$ with coordinates $\si_1\lc\si_n$.
Introduce the {discrete Wronski map\/} $\pi_\bla:\Oml\to\C^n$ as follows.
Let $X$ be a point of $\Oml$. Define
\vvn.3>
\beq
\label{wronsk of X q=1}
\Wr_X(u)\,=\,\Wr\>\bigl(f_1(u-1,X)\lc f_N(u-1,X)\bigr)\,,
\vv.3>
\eeq
where $f_1(u,X)\lc f_N(u,X)$ are given by \Ref{basis X p}. Let
\be
\Wr_X(u)\,=\,\prod_{1\le i<j\le N}(\la_j-\la_i +i-j)
\ \Bigl(u^n+\sum_{s=1}^n \>(-1)^sa_s\,u^{n-s}\Bigr)\,.
\vv.2>
\ee
We set $\pi_\bla: X\mapsto (a_1\lc a_n)$.

\vsk.2>
The discrete Wronski map is a finite algebraic map, see
\cite[Proposition 3.1]{MTV5}. It defines an injective algebra homomorphism
\vvn.3>
\be
(\pi_\bla)^* :\,\C[\si_1\dots,\si_n] \to \O_\bla ,\qquad
\si_s \mapsto \Si_s\,,
\vv.2>
\ee
which gives a $\C[\si_1\dots,\si_n]$-module structure on $\O^\qq_\bla$.

\vsk.2>
For $\aa\in\C^n$, let $I^\O_\lba$ be the ideal in \,$\Ol\>$ generated by
the elements $\Sig_s-a_s$, $s=1\lc n$, where $\Sig_1\lc\Sig_n$ are defined by
\Ref{Wr coef q=1}. The quotient algebra
\vvn.3>
\beq
\label{Olaa q=1}
\O_\lba\>=\,\Ol/I^\O_\lba
\vv.3>
\eeq
is the scheme-theoretic fiber of the discrete Wronski map $\pi_\bla$.

\subsection{Second main result}
By formula \Ref{forMulaSing}, the space
$F_{k}(\V^S)^\sing_\bla$ is one-dimensional if $k=\sum_{i=1}^N{(i-1)\la_i }$.
Fix a nonzero element \,$v^\sing_\bla \in (\V^S)^\sing_\bla$
in that subspace.

\begin{thm}
\label{2 thm}

Let $\bla$ be a partition on $n$ with at most $N$ parts. Then

\begin{enumerate}
\item[(i)]
The map \,$\mu_\bla\!: G_{k,s} \mapsto C^\bla_{k,s}$, \,$k=1\lc N$, \,$s>0$,
extends uniquely to an isomorphism
\,$\mu_\bla\!:\O_\bla\<\to\Bo\bigl((\V^S)^\sing_\bla\bigr)$ of filtered
algebras. The isomorphism \,$\mu_\bla$ becomes an isomorphism of the
\,$\C[\si_1\lc\si_n]$-module \,$\O_\bla$ and the \,$\Czs$-module
$\Bo\bigl((\V^S)^\sing_\bla\bigr)$ if we identify the algebras
\,$\C[\si_1\lc\si_n]$ and \,$\Czs$ by the map
\;$\si_s\mapsto \si_s(\bs z)$, \,$s=1\lc n$.

\vsk.2>
\item[(ii)]
The map \;$\nu_\bla: O_\bla \to\>(\V^S)^{sing}_\bla$\,,
\,$f\mapsto\mu_\bla(f)\,v_\bla^\sing$, is an isomorphism of filtered vector
spaces increasing the index of filtration by \,$\kmin$. The isomorphism
\,$\nu_\bla$ identifies the \;$\Bo\bigl((\V^S)^\sing_\bla\bigr)$-module
\;$(\V^S)_\bla^\sing$ and the regular representation of the algebra
\,$\O_\bla$.
\end{enumerate}
\end{thm}

The theorem is proved in Section \ref{proof thm main 2}.

\vsk.4>
Assume that the complex numbers \,$b_1\lc b_n$ are such that \,$b_i\ne b_j+1$
for $i>j$. Consider the tensor product $V(\bb)=\C^N(b_1)\lox\C^N(b_n)$
of evaluation $\Yn$-modules and its singular weight subspace
$V(\bb)^\sing_\bla$. Introduce the numbers $\aa=(a_1\lc a_n)$
by the formula $a_s=\si_s(b_1\lc b_n)$, cf.~\Ref{ab}.

\begin{cor}
\label{2 cor}
The isomorphisms \,$\mu_\bla$ and \,$\nu_\bla$ induce an isomorphism of the
\;$\Bo\bigl(V(\bb)^\sing_\bla\bigr)$-module \;$V(\bb)^\sing_\bla$ and the
regular representation of the algebra \,$\O_{\bla,\aa}$. In particular,\\
$\Bo\bigl(V(\bb)^\sing_\bla\bigr)\subset \End\bigl(V(\bb)^\sing_\bla\bigr)$
is a maximal commutative subalgebra and \,$\Bo\bigl(V(\bb)^\sing_\bla\bigr)$
is a Frobenius algebra, see for example \cite[Lemma 3.9]{MTV4}.
\qed
\end{cor}

\begin{rem}
Corollary \ref{2 cor} is used in \cite{MTV6} to prove Theorem 7.3 therein,
similarly to the proof of Theorems 4.1, 4.3 in \cite{MTV6}. The algebra
\,$\B^\yy_\nnla\<\>(b_1\lc b_n)$ in \cite{MTV6} coincides with the algebra
\,$\Bo\bigl(V(\bb)^\sing_\bla\bigr)$ \,in this paper.
\end{rem}

\begin{cor}
\label{2 prime cor}
Let \,$b_1,\dots,b_n$ be real and such that \,$|\>b_i-b_j|>1$ for all $i\ne j$.
Then the algebra \,$\Bo\bigl(V(\bb)^\sing_\bla\bigr)$ has simple spectrum.
\qed
\end{cor}

The proof is similar to that of Corollary \ref{1 prime cor}.

\vsk.4>
Other sufficient conditions for simplicity of the spectrum of
$\Bo\bigl(V(\bb)^\sing_\bla\bigr)$ see in \cite[Theorem~2.1, part~(2)]{MTV5}
and in \cite[Theorem~1.1]{MTV7}.

\section{Bethe ansatz for \,$\qq$ with distinct coordinates}
\label{Sec BA q ne 1}

To prove Theorems \ref{1 thm} and \ref{2 thm} we need some facts about
the Bethe ansatz. We consider the tensor product of evaluation $\Yn$-modules
$V(\bb)=\C^N(b_1)\lox\C^N(b_n)$ and the action of the Bethe algebra \,$\Bcq$
on the weight subspace $V(\bb)_\bla$.

\subsection{Bethe ansatz equations associated with $V(\bb)_\bla$}
\label{Bethe ansatz equations associated with a weight subspace}

Recall \,$\bla=(\la_1\lc\la_N)$, \,$|\bla|=n$. Introduce
\,$\bs l = (l_1\lc l_{N-1})$ with $l_j=\,\la_{j+1}\lsym+\la_N$.
\vvgood
We have $n\geq l_1\geq \dots \geq l_{N-1}\geq 0$.
Set \,$l_0=n$\>, \,$l_N=0$\>, and \,$l=l_1 + \dots + l_{N-1}$.
We shall consider functions of \,$l$ variables
\vvn.3>
\be
\bs t = \big(t^{(1)}_{1}\lc t^{(1)}_{l_1},
t^{(2)}_{1}\lc t^{(2)}_{l_2}\lc t^{(N-1)}_{1}\lc t^{(N-1)}_{l_{N-1}}\big)\,.
\vv.3>
\ee
The following system of \,$l$ algebraic equations with respect to \,$l$
variables $\bs t$ is called the Bethe ansatz equations associated with
$V(\bb)_\bla$ and $\qq$\,:
\vvn.3>
\begin{align}
\label{BAE}
q_1\,\prod_{s=1}^n\,(t^{(1)}_j - b_s + 1) &\,
\prod_{\satop{j'=1}{j'\neq j}}^{l_1}\,(t^{(1)}_j - t^{(1)}_{j'} - 1 )\,
\prod_{j'=1}^{l_2} (t^{(1)}_j - t^{(2)}_{j'})\,={}
\\
{}=\,q_2\,\prod_{s=1}^n (t^{(1)}_j - b_s) &\,
\prod_{\satop{j'=1}{j'\neq j}}^{l_1}\,(t^{(1)}_j - t^{(1)}_{j'} + 1 )\,
\prod_{j'=1}^{l_2} (t^{(1)}_j - t^{(2)}_{j'} - 1)\,,
\notag
\end{align}
\vv-.5>
\begin{align*}
q_a\,\prod_{j'=1}^{l_{a-1}}\,(t^{(a)}_j - t^{(a-1)}_{j'} + 1) &\,
\prod_{\satop{j'=1}{j'\neq j}}^{l_a}\,(t^{(a)}_j - t^{(a)}_{j'} - 1 )\,
\prod_{j'=1}^{l_{a+1}} (t^{(a)}_j - t^{(a+1)}_{j'})\,={}
\\
{}=\,q_{a+1}\,\prod_{j'=1}^{l_{a-1}} (t^{(a)}_j - t^{(a-1)}_{j'}) &\,
\prod_{\satop{j'=1}{j'\neq j}}^{l_a}\,(t^{(a)}_j - t^{(a)}_{j'} + 1 )\,
\prod_{j'=1}^{l_{a+1}} (t^{(a)}_j - t^{(a+1)}_{j'}-1)\,.
\\[-16pt]
\end{align*}
Here the equations of the first group are labeled by \,$j=1\lc l_1$, the
equations of the second group are labeled by \,$a=2\lc N-1$, \,$j=1\lc l_a$,

\vsk.2>
A solution \;$\ttbi$ of system \Ref{BAE} is called off-diagonal if
\;${\tti^{(a)}_{j} \neq \tti^{(a)}_{j'}}$ \,for any \,$a=1\lc\alb N-1$,
\,$1\leq j \leq j' \leq l_a$, \,and \,$\tti^{(a)}_{j}\neq\tti^{(a+1)}_{j'}$
for any \,$a=1\lc N-2$, \,$j= 1\lc l_a$, \,$j'= 1\lc l_{a+1}$.

\begin{rem}
If \,$\bla=(n,0\dots,0)$, then \,$l_1=\dots=l_{N-1}=0$. In this case, there
are no variables $\TT$ and it is convenient to think that the Bethe ansatz
equations is the equation $1=1$.
\end{rem}

\subsection{Weight function and Bethe ansatz theorem}
\label{Weight function and}

Denote by \,$\om_\bla(\bs t,\bb)$ the universal weight function associated
with the weight subspace $V(\bb)_\bla$. The universal weight function is
defined by formula~(6.2) in \cite{MTV1}, see explicit formula~\Ref{hWI-}
below, cf.~\cite{TV1, MTV2, RTV, TV3}. For the moment, it is enough for us
to know that this function is a $V(\bb)_\bla$-valued polynomial in $\bs t$,
\,$\bb$.

If $\tilde{\bs t}$ is an off-diagonal solution of the Bethe ansatz equations,
then the vector $\om_\bla(\ttbi,\bb) \in V(\bb)_\bla$ is called the Bethe
vector associated with $\tilde{\bs t}$.

\begin{thm}
\label{thm on Bethe ansatz}
Let \;$\ttbi$ be an off-diagonal solution of the Bethe ansatz equations
\Ref{BAE}. Assume that the Bethe vector \,$\om_\bla(\ttbi,\bb)$ is nonzero.
Then the Bethe vector is an eigenvector of all transfer-matrices \,$\Bq_k(u)$,
\,$k=1\lc N$.
\qed
\end{thm}

The statement follows from Theorem~6.1 in~\cite{MTV1}.
For ${k=1}$, the result is established in~\cite{KR1}.

\vsk.2>
The eigenvalues of the Bethe vector are as follows. Set
\vvn.3>
\begin{align*}
& \chi_1(u, \bs t,\bb)\,=\,q_1\, \prod_{s=1}^n\,\frac{u-b_s + 1}{u-b_s}
\ \prod_{j=1}^{l_1}\,\frac{u- t^{(1)}_j - 1}{u-t^{(1)}_j}\;,
\\[6pt]
& \chi_a(u, \bs t,\bb)\,=\,
q_a \,\prod_{j=1}^{l_{a-1}}\,\frac{u- t^{(a-1)}_j + 1}{u-t^{(a-1)}_j}
\ \prod_{j=1}^{l_a}\,\frac{u- t^{(a)}_j - 1}{u-t^{(a)}_j}\;,
\\[-14pt]
\end{align*}
for \,$a=2\lc N$. Define the functions \,$c_k(u,\bs t,\bb)$ \>by the formula
\beq
\label{chI}
(1-\chi_1(u,\bs t,\bb)\,\tau)\cdots
(1-\chi_N(u,\bs t,\bb)\,\tau)\,=\,
\sum_{k=0}^N \,(-1)^k\,c_k(u,\bs t,\bb)\,\tau^k\>.
\vv-.6>
\eeq
Then
\beq
\label{eig v}
\Bq_k(u)\,\om_\bla(\ttbi,\bb)\,=\,
c_k(u,\ttbi,\bb)\,\om_\bla(\ttbi,\bb)
\vv.2>
\eeq
for \,$k=1\lc N$, see Theorem 6.1 in \cite{MTV1}.

\begin{rem}
If $\bla=(n,0\dots,0)$, then $V(\bb)_\bla$ is the one-dimensional space
generated by the vector $v_1\otimes\dots\otimes v_1$. It is convenient
to assume that the universal weight function is given by the formula
\,$\om(\bb)=v_1\lsym\otimes v_1$. This vector is an eigenvector of the Bethe
algebra \,$\Bcq$. The eigenvalues of this eigenvector are defined by formula
\Ref{chI}, in which the difference operator takes the form
\vvn-.4>
\beq
\label{TR BA}
\biggl(1- q_1\,\Bigl(\,\prod_{s=1}^n \frac{u-b_s+1}{u-b_s}\>\Bigr)\,\tau
\biggr)\,\prod_{i=2}^N\,(1 - q_i\,\tau)\,.
\vv.3>
\eeq
\end{rem}

\subsection{Difference operator associated with an off-diagonal solution}
\label{DI off}

For $\ttbi\in\C^l$, we introduce the associated fundamental difference
operator
\vvn.3>
\beq
\label{fund op}
\Dt\,=\,\sum_{k=0}^N \,(-1)^k\,c_k(u,\ttbi,\bb)\,\tau^k\>,
\vv.1>
\eeq
see \cite{MTV2}. Here the functions \,$c_k(u,\ttbi,\bb)$ \>are given
by \Ref{chI}.

\begin{thm}
\label{thm on fund operator}
Assume that \,$\qq$ has distinct coordinates. Let \;$\ttbi$ be an off-diagonal
solution of the Bethe ansatz equations. Then there exist polynomials
\,$p_k(u)$\>, \,$k=1\lc N$, such that \;$\deg\>p_k(u)=\la_k$\>,
\,\,$\D_{\ttbi}(u,\tau)\,q_k^u\,p_k(u)=0$\>, and
\beq
\label{Wr b}
\Wr\>(q_1^{u-1}p_1(u-1)\lc q_N^{u-1}p_N(u-1))\,=\,
\prod_{i=1}^N q_i^{u-1}\!\prod_{1\le i<j\le N}(q^{-1}_j-q^{-1}_i)\,
\prod_{s=1}^n\,(u-b_i)\,.\kern-1em
\vv-.8>
\eeq
\qed
\end{thm}

This is Proposition~7.6 in \cite{MV3}, which is a generalization of
Lemma~4.8 in \cite{MV1}.

\begin{rem}
If \,$\bla=(n,0\dots,0)$\>, then $V(\bb)_\bla$ is spanned by
\,$v_1\lsym\otimes v_1$. The corresponding fundamental difference
operator \,$\Dt(u,\tau)$ is given by \Ref{TR BA}. The polynomials
\,$p_2(u)\lc p_N(u)$ of Theorem \ref{thm on fund operator} are just constants
and the polynomial \,$p_1(u)$ is uniquely determined (up to proportionality)
by the condition \,$\Dt(u,\tau)\,q_1^u\,p_1(u)=0$.
\end{rem}

\subsection{Completeness of the Bethe ansatz}

\begin{thm}
\label{thm on completeness}
Assume that \,$\qq$ has distinct coordinates. Then for any \,$\bs\la$ and
generic \,$b_1\lc b_n$, there exists a collection of off-diagonal solutions
of the Bethe ansatz equations such that the corresponding Bethe vectors form
a basis of \;$V(\bb)_\bla$.
\end{thm}

Theorem \ref{thm on completeness} is proved in Sections \ref{sec pr n=1}
and \ref{sec n>1}.

\subsection{Weight functions \,$W_I$\>}
For a function $f(t_1\lc t_k)$ of some variables, denote
\vv.3>
\be
\Sym_{t_1\lc t_k}\>f(t_1\lc t_k)\,=\,
\sum_{\si\in S_k\!}\,f(t_{\si_1}\lc t_{\si_k})\,.
\vv-.2>
\ee
Recall \,$\bla=(\la_1\lc\la_N)$ and \,$I=(I_1\lc I_N)$. Set
\;$\bigcup_{\>c=a+1}^{\,N}I_c=\>\{\>i^{(a)}_1\!\lsym<i^{(a)}_{l_a}\}$\>.
Introduce $t^{(0)}=(t^{(0)}_1\lc t^{(0)}_n) = (b_1\lc b_n)$\>.

\vsk.2>
For $I\in \Il$\>, we define the weight functions \,$W_{I}(\TT;\bb)$\>,
cf.~\cite{TV1, TV3}:
\vvn.4>
\beq
\label{hWI-}
W_I(\TT;\bb)\,=\,
\Sym_{\>t^{(1)}_1\!\lc\,t^{(1)}_{l_1}}\,\ldots\;
\Sym_{\>t^{(N-1)}_1\!\lc\,t^{(N-1)}_{l_{N-1}}}\,U_I(\TT;\bb)\,,
\vv.3>
\eeq
\be
U_I(\TT;\bb)\,=\,\prod_{a=1}^{N-1}\,\prod_{j=1}^{l_a}\,\biggl(
\prod_{\satop{j'=1}{i^{(a-1)}_{j'}\<<\>i^{(a)}_j}}^{l_{a-1}}
\!\!(t^{(a)}_j\<\<-t^{(a-1)}_{j'}+1)
\prod_{\satop{j'=1}{i^{(a-1)}_{j'}>\>i^{(a)}_j}}^{l_{a-1}}
\!\!(t^{(a)}_j\<\<-t^{(a-1)}_{j'})\,\prod_{j'=j+1}^{l_a}
\frac{t^{(a)}_j\<\<-t^{(a)}_{j'}\<\<+1}{t^{(a)}_j\<\<-t^{(a)}_{j'}}\,\biggr)\,.
\vv.3>
\ee
The universal $V(\bb)_\bla$-valued weight function is the function
\vvn.4>
\beq
\label{UNI w}
\om_\bla(\TT,\bb)\,=\,
\sum_{I\in\Il}\,W_I(\TT,\bb)\,v_I\,\in V(\bb)_\bla\,.
\eeq

\subsection{Proof of Theorem \ref{thm on completeness} for $n=1$}
\label{sec pr n=1}

If \,$n=1$, then \,$\bla=(0\lc 0, 1_{k+1},0\lc 0)$\>, where \,$1$ \,is at
the \,$k+1$-st position. If \,$k=0$, Theorem \ref{thm on completeness} holds
due to remarks in Sections \ref{Bethe ansatz equations associated with a weight
subspace} and \ref{Weight function and}.

Assume \,$k>0$. Then \,$\TT=(t^{(1)}_1,t^{(2)}_1\lc t^{(k)}_1)$ \,and
\,$\om_\bla(\TT,\bb) = v_{k+1}$. The Bethe ansatz equations are
\begin{align}
\label{BAE 1}
q_1\,(t^{(1)}_1 - b_1 + 1)\>(t^{(1)}_1 - t^{(2)}_{1})\, &{}=\,
q_2 (t^{(1)}_1 - b_1)(t^{(1)}_1 - t^{(2)}_{1} - 1)\,,
\\[4pt]
q_a\,(t^{(a)}_1 - t^{(a-1)}_{1} + 1)\>(t^{(a)}_1 - t^{(a+1)}_{1})\, &{}=\,
q_{a+1}\,(t^{(a)}_1 - t^{(a-1)}_{1})\>(t^{(a)}_1 - t^{(a+1)}_{1}-1)\,,
\notag
\\[4pt]
q_k\,(t^{(k)}_1-t^{(k-1)}_{1} + 1)\,&{}=\,q_{k+1}\,(t^{(k)}_1-t^{(k-1)}_{1})\,.
\notag
\\[-14pt]
\notag
\end{align}
Here the equations of the second group are labeled by \,$a=2\lc k-1$\>.
The Bethe ansatz equations have the unique solution
\beq
\label{Unique soln}
t^{(i)}_1=\,b_1\>+\>\sum_{j=1}^i\,\frac{q_j}{q_{k+1}-q_j}\,,\qquad i=1\lc k\,.
\vv.2>
\eeq
This solutions is off-diagonal. Theorem \ref{thm on completeness} for \,$n=1$
\,is proved.

\subsection{Proof of Theorem \ref{thm on completeness} for $n>1$}
\label{sec n>1}
Assume that $b_1\lc b_n$ depend on a parameter $y\in\C$, so that $b_s(y)=sy$\>.
The next lemma implies Theorem \ref{thm on completeness}.

\begin{lem}
\label{lem BA}
For \,$I\in\Il$\>, there exists an off-diagonal solution \;$\ttbi(y)$ of
the Bethe ansatz equations \Ref{BAE} such that the line generated by the Bethe
vector \,$\om\bigl(\ttbi(y),\bb(y)\bigr) $ tends to the line generated by
the vector \,$v_I$ as \,$y$ tends to infinity.
\end{lem}

\begin{proof}
To simplify the notations we consider an example. The general case is similar.
Assume that \,$n=2$ and \,$v_I=\>v_3\otimes v_2$. Then
\,$\TT=(t^{(1)}_1,t^{(1)}_2,t^{(2)}_1)$. We look for a solution of the Bethe
ansatz equations in the form
\vvn.3>
\beq
\label{form}
t^{(i)}_1 =\,b_1(y) + v^{(i)}_1(y)\,,\quad i=1,2\,,\qquad \text{and}\qquad
t^{(1)}_2 =\,b_2(y) + v^{(1)}_2(y)\,.\kern-2em
\vv.3>
\eeq
Then the Bethe ansatz equations take the form
\vv.3>
\begin{gather}
\label{BAE 3}
\frac{v^{(1)}_2 + 1}{v^{(1)}_2}\,=\,\frac{q_2}{q_1} + O(y^{-1})\,,\qquad
\frac{v^{(1)}_1 + 1}{v^{(1)}_1}
\cdot\frac {v^{(1)}_1 - v^{(2)}_{1}}{v^{(1)}_1 - v^{(2)}_{1} - 1}\,=\,
\frac{q_2}{q_1} + O(y^{-1})\,,
\\
\frac{v^{(2)}_1 - v^{(1)}_{1} + 1}{v^{(2)}_1 - v^{(1)}_{1}}\,=\,
\frac{q_{3}}{q_{2}} + O(y^{-1})\,.
\notag
\end{gather}
As $y\to\infty$\>, this system of three equations splits into an equation
assigned to \,$b_1$ and a system of two equations assigned to \,$b_2$ according
to our choice \Ref{form}. Each of the limiting systems is the system
\Ref{BAE 1} of the Bethe ansatz equations for \,$n=1$ \,considered in
Section~\ref{sec pr n=1}. System \Ref{BAE 1} has a unique solution
\Ref{Unique soln}. By deforming that solution, we obtain a solution $\vvbi(y)$
of system \Ref{BAE} whose limit as \,$y\to\infty$ \,equals
\vvn.1>
\be
\Bigl(\,\frac{q_1}{q_3-q_1}\,,\,\frac{q_1}{q_2-q_1}\,,\,
\frac{q_1}{q_3-q_1}+\frac{q_2}{q_3-q_2}\,\Bigr)\,,
\vv.1>
\ee
see Section \ref{sec pr n=1}. Clearly, \,$\vvbi(y)$ corresponds to an
off-diagonal solution $\ttbi(y)$ of the Bethe ansatz equations as $y\to\infty$.
It is easy to see that the limit of the line generated by the Bethe vector
\,$\om\bigl(\ttbi(y),\bb(y)\bigr)$ as \,$y\to\infty$ is the line generated by
the vector \,$v_3\otimes v_2$, see formula \Ref{hWI-}.
\end{proof}

\section{Bethe ansatz for \,$\qq=\bs 1$}
\label{Sec BA q 1}

We consider the action of the Bethe algebra \,$\Bo$ on the subspace
$V(\bb)^\sing_\bla$ of singular vectors. We use the notation of
Section \ref{Sec BA q ne 1}.

\vsk.2>
For \,$\qq=\bs 1$, the Bethe ansatz equations are given by \Ref{BAE}
with $q_1\lsym=q_N=1$. Let \,$\om_\bla (\bs t,\bb)$ be the universal weight
function associated with the weight subspace $V(\bb)_\bla$, see \Ref{UNI w}.

\begin{thm}
\label{thm on Bethe ansatz q=1}
Let \;$\ttbi$ be an off-diagonal solution of the Bethe ansatz equations for
\,$\qq=\bs 1$. Then \,$\om_\bla(\ttbi,\bb)$ lies in \,$V(\bb)^\sing_\bla$.
\qed
\end{thm}

The statement is Lemma~5.3 in \cite{TV2} or Proposition~6.2 in \cite{MTV1}.
For $N=2$, the result is established in~\cite{FT}, and for $N=3$ in~\cite{KR2}.

\begin{thm}[{\cite[Lemma~4.8]{MV1}}]
\label{thm on fund operator q=1}
Let \;$\ttbi$ be an off-diagonal solution of the Bethe ansatz equations for
\,$\qq=\bs 1$. Consider the the associated fundamental difference operator
\,$\Dt$, see \Ref{fund op}. There exist polynomials $f_k(u)\in\C[u]$\>,
\,$k=1\lc N$, of the form described in \Ref{basis p}, such that
\;$\deg f_k(u) = \la_k + N-k$\>, \,\,$\Dt(u,\tau)f_k(u) = 0$\>, \,and
\beq
\label{Wr b q=1}
\Wr\>\bigl(f_1(u-1)\lc f_N(u-1)\bigr)\,=\,
\prod_{1\le i<j\le N}(\la_j-\la_i +i-j)\;\prod_{s=1}^n\,(u-b_i)\,.\kern-2em
\vv-1.4>
\eeq
\qed
\end{thm}

\begin{thm}
\label{thm on compl q=1}
For generic \;$b_1\lc b_n$ there exists a collection of off-diagonal solutions
of the Bethe ansatz equations for $\qq=\bs 1$ such that the corresponding Bethe
vectors form a basis of $V(\bb)^\sing_\bla$.
\end{thm}
\begin{proof}
Assume that $b_1\lc b_n$ depend on a parameter $y\in\C$, so that $b_s=yd_s$
for given \,$d_1\lc d_n$, and \,$y$ tends to infinity. We look for a solution
of equations \Ref{BAE} for \,$\qq=\bs 1$ in the form
\,$t^{(i)}_j=y\>v^{(i)}_j$. Then the equations are
\beq
\label{BAE 5}
\sum_{s=1}^n\,\frac{1}{v^{(1)}_j - d_s}\>-\>
\sum_{\satop{j'=1}{j'\neq j}}^{l_1}\,\frac2{v^{(1)}_j - v^{(1)}_{j'}}\>+\>
\sum_{j'=1}^{l_2} \frac {1}{v^{(1)}_j - v^{(2)}_{j'}}\,=\,0+O(y^{-1})\,,
\eeq
\be
\sum_{j'=1}^{l_{a-1}}\,\frac{ 1}{v^{(a)}_j - v^{(a-1)}_{j'}}\>-\>
\sum_{\satop{j'=1}{j'\neq j}}^{l_a}\,\frac2{v^{(a)}_j - v^{(a)}_{j'} }\>+\>
\sum_{j'=1}^{l_{a+1}}\,\frac{1}{v^{(a)}_j - v^{(a+1)}_{j'}}\,=\,0+O(y^{-1})\,.
\vv.3>
\ee
In the limit \,$y\to\infty$\>, equations \Ref{BAE 5} tend to the Bethe ansatz
equations of the Gaudin model associated with $V({\bs d})_\bla^\sing$,
considered in \cite{MV2}. By \cite[Theorem 6.1]{MV2}, for generic
\,$\bs d=(d_1\lc d_n)$ \>there exists a collection of off-diagonal
multiplicity-free solutions \,$\vvbi$ of system \Ref{BAE 5} at \,$y=\infty$
\>such that the corresponding Bethe vectors \,$\om_\bla^\Gaud(\vvbi,{\bs d})$
of the Gaudin model, see formula~(4) in \cite{MV2}, form a basis of
$V^\sing_\bla$. By deforming these solutions, we get a collection of solutions
\,$\vvbi(y)$ of system \Ref{BAE 5}, and hence a collection of off-diagonal
solutions \,$\ttbi(y)= y\>\vvbi(y)$ of system \Ref{BAE} for \,$\qq=\bs 1$.
It easily follows from formula \Ref{hWI-} and formula~(4) in \cite{MV2} that
the lines generated by the Bethe vectors \,$\om_\bla(\ttbi(y),\bb(y))$ tend
respectively to the lines generated by the Bethe vectors
\,$\om_\bla^\Gaud(\vvbi,{\bs d})$. Hence the Bethe vectors
\,$\om_\bla(\ttbi(y),\bb(y))$ form a basis of \,$V(\bb)_\bla^\sing$
as \,$y\to\infty$. This proves Theorem \ref{thm on compl q=1}.
\end{proof}

\section{Proofs of Theorems \ref{1 thm} and \ref{2 thm}}
\label{sec proofs}

\subsection{Proof of Theorem \ref{1 thm}}
\label{proof thm main 1}

Let a polynomial \,$R(\Fql_{k,s})$ in generators \,$\Fql_{k,s}$ equal zero
in \,$\Oql$. Consider $R(\Bql_{k,s})$ as a polynomial in $\zzz$ with
values in $\End\bigl((V^{\otimes n})_\bla\bigr)$. By Theorems
\ref{thm on fund operator} and \ref{thm on completeness}, this polynomial
equals zero for generic values of \,$\zzz$. Hence, \,$R(\Fql_{k,s})$
equals zero identically and the map
\,$\mu^\qq_\bla:\Oql\<\to\Bcq\bigl((\V^S)_\bla\bigr)$ is well-defined.

\vsk.2>
Let a polynomial \,$R(\Fql_{k,s})$ be a nonzero element of \,$\Oql$.
By Theorems \ref{thm on fund operator} and \ref{thm on completeness},
it means that \,$R(\Bql_{k,s})$ is nonzero in \,$\Bcq\bigl((\V^S)_\bla\bigr)$.
This shows that \,$\mu^\qq_{\bs\la}$ is injective. Since the elements
\,$\Bql_{k,s}$ generate the algebra $\Bcq\bigl((\V^S)_\bla\bigr)$, the map
\,$\mu^\qq_{\bs\la}$ is surjective. By comparing formulae \Ref{BN} and
\Ref{FN}, we conclude that \,$\mu^\qq_{\bs\la}$ is a homomorphism of
the \,$\C[\si_1\lc\si_n]$-module \,$\Oql$ to the \,$\Czs$-module
$\Bcq\bigl((\V^S)_\bla\bigr)$. Since \;$\deg\>\Fql_{k,s} =\deg\>\Bq_{k,s}$\>,
the homomorphism \,$\mu^\qq_{\bs\la}$ is filtered. These remarks prove
part (i) of Theorem \ref{1 thm}.

\vsk.2>
Consider the map \,$\nu^\qq_\bla:\Oql\to\>(\V^S)_\bla$,
$f\mapsto\mu^\qq_\bla(f)\,v_\bla$. The kernel of \,$\nu^\qq_\bla$ is an ideal
in $\Oql$ which has zero intersection with \,$\C[\si_1\lc\si_n]$ \,and,
therefore, is the zero ideal. Since \,$\ch_{\>\Oql}(t)=\ch_{\>(\V^S)_\bla}(t)$,
we conclude that \,$\nu^\qq_\bla$ is a linear isomorphism. This gives part (ii)
of Theorem \ref{1 thm}.
\qed

\subsection{Proof of Theorem \ref{2 thm}}
\label{proof thm main 2}
The proof of Theorem \ref{2 thm} is similar to the proof of
Theorem \ref{1 thm}.

\vsk.2>
Namely, let a polynomial \,$R(G_{k,s})$ in generators \,$G_{k,s}$ equal zero
in $\Ol$. Consider \,$R(C_{k,s})$ as a polynomial in \,$\zzz$ with values
in $\End\bigl((V^{\otimes n})^\sing_\bla\bigr)$. By Theorems
\ref{thm on fund operator q=1} and \ref{thm on compl q=1}, this polynomial
equals zero for generic values of \,$\zzz$. Hence, \,$R(C_{k,s})$ equals
zero identically and the map $\mu_\bla:\Ol\to\Bo\bigl((\V^S)^\sing_\bla\bigr)$
is well-defined.

\vsk.2>
Let a polynomial \,$R(G_{k,s})$ be a nonzero element of \,$\Ol$.
By Theorems \ref{thm on fund operator q=1} and \ref{thm on compl q=1},
it means that \,$R(C_{k,s})$ is nonzero in $\Bo\bigl((\V^S)^\sing_\bla\bigr)$.
This shows that \,$\mu_\bla$ is injective. Since the elements \,$C_{k,s}$
generate the algebra \,$\Bo\bigl((\V^S)^\sing_\bla\bigr)$, the map
\,$\mu_\bla$ is surjective. By comparing formulae \Ref{BN} and \Ref{FN},
we conclude that \,$\mu_\bla$ is a homomorphism of the
\,$\C[\si_1\lc\si_n]$-module \,$\O_\bla$ to the \,$\Czs$-module
\,$\Bo\bigl((\V^S)^\sing_\bla\bigr)$. Since \;$\deg\>G_{k,s} =\deg\>C_{k,s}$,
the homomorphism \,$\mu_\bla$ is filtered. These remarks prove part (i) of
Theorem \ref{2 thm}.

\vsk.2>
Consider the map \,$\nu_\bla:\Ol\to\>(\V^S)^\sing_\bla$,
$f \mapsto \mu_\bla(f)\,v_\bla$. The kernel of $\nu_\bla$ is an ideal
in \,$\Ol$ which has zero intersection with \,$\C[\si_1\lc\si_n]$ and,
therefore, is the zero ideal. Since
\,$t^{\sum_{i=1}^N(i-1)\la_i}\ch_{\>\O_\bla}(t)=\ch_{\>(\V^S)^\sing_\bla}(t)$,
we conclude that \,$\nu_\bla$ is a linear isomorphism. This gives part (ii) of
Theorem \ref{2 thm}.
\qed

\section{Space \,$\DVe$ }
\label{alg sec A}

\subsection{Definitions}
\label{sec S-actions A}
Recall $\V=V\otimes \C[z_1,\dots,z_n]$ and the $S_n$-action on $V\<$-valued
functions of $z_1,\dots,z_n$ defined by formula \Ref{Sn+}. We denote by $\V^A$
the subspace of the $S_n$-skew-invariants in $\V$. The space $\V^A$ is a
filtered space.

Let
\be
D\,=\!\prod_{1\le i<j\le n}\!(z_j - z_i + 1)\,.
\vv.3>
\ee
We denote by $\Di\V$ the space of $V\<$-valued functions of $z_1,\dots,z_n$
of the form $\Di f$, \,$f\<\in\V$, and by $\DVe$ the space of $V\<$-valued
functions of $z_1,\dots,z_n$ of the form $\Di f$, \,$f\<\in\V^A$.

\begin{lem}[{\cite[Lemma 2.9]{GRTV}}]
\label{S4}
A $V$-valued function $f$ of $z_1,\dots,z_n$ is skew-invariant with respect
to the $S_n$-action if and only if the function \,$\Di f$ \,is invariant with
respect to the $S_n$-action. \qed
\end{lem}

By this lemma, we can define the space $\DVe$ as the space of $V\<$-valued
$S_n$-invariant functions of $z_1,\dots,z_n$ of the form $\Di f$, \,$f\<\in\V$.

\begin{lem}
\label{VSfree A}
The space $\DVe$ is a free $\Czs$-module of rank $N^n$.
\end{lem}

\begin{proof}
The lemma follows from Lemma 2.10 in \cite{GRTV}.
\end{proof}

We define the degree of elements $\Di f\in\DVe$ by the formula
$\deg(\Di f)=\deg(f)-n(n-1)/2$. We consider the increasing filtration
$\dots\subset F_{k-1}\DVe\subset F_{k}\DVe \subset \dots \subset \DVe$
whose $k$-th subspace consists of elements of degree $\leq k$.

The space \,$\DVe\<$ is a filtered
$\gln$-module. We consider the $\gln$-weight decomposition
\vvn-.1>
\be
\DVe=\!\tbigoplus_{\fratop{\bla\in\Z^N_{\geq 0}}{|\bla|=n}}\!(\DVe)_\bla\,,
\vv.2>
\ee
as well as the subspaces of singular vectors $(\DVe)^\sing_\bla\!\subset
(\DVe)_\bla$. All of these are filtered free $\Czs$-modules.

\begin{lem}
\label{lem on char of VSl A}
For $\bs\la\in\Z^N_{\geq 0}$, $|\bla|=n$, we have
\vvn-.2>
\beq
\label{forMula A}
\on{ch}_{\>\DLe}(t)\,=\,t^{\,-\sum_{1\le i<j\le N}\la_i\la_j}\;
{\prod_{i=1}^N\,\frac1{(\<\>t\<\>)_{\la_i}}}\;.
\eeq
For a partition $\bla$ of $n$ with at most $N$ parts, we have
\vvn.2>
\beq
\label{forMulaSing A}
\on{ch}_{\>\DLs}(t)\,=\,t^{\,-\sum_{1\le i<j\le N}{\la_i\la_j}}\;
\frac{\prod_{1\le i<j\le N}\,(1-t^{\la_i-\la_j+j-i})}
{\prod_{i=1}^N\,(\<\>t\<\>)_{\la_i+N-i}}\;.
\eeq
\end{lem}
\begin{proof}
The graded components $F_k\DLe/F_{k-1}\DLe$ and $F_k\DLs/F_{k-1}\DLs$
are respectively naturally isomorphic to the graded components considered
in \cite{RSTV} and denoted there by $((\frac 1D\V^-)_\bla)_k$ and
$((\Di\V^-)^{sing}_\bla)_k$. Now formula \Ref{forMula A} follows from
Theorem \cite[Theorem 3.4]{RSTV} and \cite[Lemma 2.12]{MTV3}.
Formula \Ref{forMulaSing A} follows from \cite[Formula 3.4]{RSTV}.
\end{proof}

\subsection{Space $\DVe$ as a Yangian module}

By Lemmas \ref{lem filt} and \ref{pm & Yan 1}, $\DVe$ is a filtered
$\Yn$-module.

For \,$\bb=(b_1\lc b_n)\in\C^n$, we define \,$\aa=(a_1\lc a_n)$ by the formula
$a_s=\si_s(b_1\lc b_n)$, cf.~\Ref{ab}.

\begin{prop}
\label{VAVb}
Assume that the numbers \,$b_1\lc b_n$ are such that \,$b_i\ne b_j+1$ for all
$i\ne j$. Then the $\Yn$-module $\DVe\<\</I_\aa\DVe$ is isomorphic to
$V(\bb)=\C^N(b_1)\lox\C^N(b_n)$, the tensor product of evaluation
$\Yn$-modules.
\end{prop}
\begin{proof}
Consider the map \,$\phi^A\!:\DVe\!\to V(\bb)$ that sends every element of
\,$\DVe$ to its value at the point \,$\zb=\bb$\>. This map is a homomorphism
of $\Yn$-modules and factors through the canonical projection
\,$\thi^A\!:\DVe\!\to\DVe\<\</I_\aa\DVe$. Since \,$\thi^A\<$ is also
a homomorphism of $\Yn$-modules, this defines a homomorphism of $\Yn$-modules
\,$\psi^A\!:\DVe\<\</I_\aa\DVe\!\to V(\bb)$.

\vsk.2>
Under the assumption that \,$b_i\ne b_j+1$ for \,$i\ne j$\>, the $\Yn$-module
\vvn.1>
\,$V(\bb)$ is irreducible, see Proposition \ref{irr}. Since $\psi(\voxn)=\voxn$
\,and
\vvn.1>
\;$\dim\<\>\DVe\<\</I_\aa\DVe\<=N^n\<=\dim\<\>V(\bb)$, \,the map \,$\psi$
\,is an isomorphism of $\Yn$-modules.
\end{proof}




\medskip
The Bethe algebra $\B^\qq$ preserves the subspaces $\DLe\subset\DVe$. The image
of an element $B^\qq_{k,s}\in\B^\qq$ in $\End(\DLe)$ will be denoted by
$\tilde{B}^{\qq,\bla}_{k,s}$.

\medskip
The Bethe algebra $\B^{\qq=\bs 1}$ preserves the subspaces $\DLs\subset\DVe$. Recall the elements
$C_{k,s}\in \B^{\qq=\bs 1}$ introduced by formula \Ref{formula for Se}. The image
of an element $C_{k,s}\in\B^\qq$ in $\End(\DLs)$ will be denoted by
$\tilde{C}^{\bla}_{k,s}$.

\begin{thm}[{\cite[Theorem 3.7]{MTV2}}]
\label{thm Q=1 A}
The elements $\tilde{C}^{\bla}_{1,1},\dots,\tilde{C}^{\bla}_{N,N}$ are scalar operators, and for a variable $x$, we have
\be
\sum_{k=0}^N\,(-1)^k\,\tilde{C}^{\bla}_{k,k}\prod_{j=0}^{N-k-1}\!(x-j)\,=\,
\prod_{s=1}^N\,(x-\la_s-N+s)\,.
\vv-1.3>
\ee
\qed
\end{thm}

\subsection{Third main result}

By formula \Ref{forMula A}, the space
$F_k\DLe$ is one-dimensional if $k=-\sum_{1\le i<j\le N}\la_i\la_j$.
We fix a nonzero element $v^A_\bla$ of that space. If $\bla$ is such that
$\la_1\geq\dots\geq\la_N$, then the element $v^A_\bla$ lies in the one-dimensional
space $F_k\DLs$, see \Ref{forMulaSing A}.

The properties of the element $v^A_\bla$ were discussed in \cite {RTVZ}. In particular, see there a
geometric description of $v^A_\bla$ in terms of orbital varieties. The element $v^A_\bla$ was denoted
in \cite{GRTV} by $v^=_\bla$, see \cite[Formula 2.27]{GRTV}.

\begin{thm}
\label{3 thm}
Assume that $\qq\in(\C^\times)^N$ has distinct coordinates.
Then
\begin{enumerate}
\item[(i)]
The map $\tilde \mu^\qq_\bla\!:\Fql_{k,s} \mapsto \tilde{B}^{\qq,\bla}_{k,s}$,
\,$k=1\lc N$, \,$s>0$, extends uniquely to an isomorphism
$\tilde \mu^\qq_\bla\!: \O^\qq_\bla\<\to\Bcq\bigl(\DLe\bigr)$
of filtered algebras. The isomorphism $\tilde\mu^\qq_\bla$ becomes an isomorphism of
the $\C[\si_1\lc\si_n]$-module \,$\O^\qq_\bla$ and the
$\Czs$-module \,$\Bcq\bigl(\DLe\bigr)$ if we identify the algebras
$\C[\si_1\lc\si_n]$ and $\Czs$ by the map
$\si_s\mapsto \si_s(\bs z)$, \,$s=1\lc n$.
\vsk.2>
\item[(ii)]
The map \;$\tilde\nu^\qq_\bla:\O^\qq_\bla \to\>\DLe$\,,
\,$f\mapsto\tilde\mu^\qq_\bla(f)\,v^A_\bla$\>,
is an isomorphism of filtered vector spaces identifying the
\;$\Bcq\bigl(\DLe\bigr)$-module \;$\DLe$ and the regular
representation of \,$\O^\qq_\bla$.
\end{enumerate}
\end{thm}

\begin{proof}
The proof of Theorem \ref{3 thm} word by word coincides with that of
Theorem \ref{1 thm}.
\end{proof}

\begin{rem}
Theorem \ref{3 thm} was announced in \cite[Theorem 6.4]{GRTV}, cf.~the remark after Theorem \ref{1 thm}.
As explained in \cite{GRTV}, Theorem 6.4 in \cite{GRTV} implies Theorem 6.5 in \cite{GRTV}.

\end{rem}

\begin{cor}
\label{cor on iso}
The $\B^\qq$-modules $(\V^S)_\bla$ and $\DLe$ are isomorphic.
\end{cor}

\begin{proof} The corollary follows from Theorems \ref{1 thm} and \ref{3 thm}.
\end{proof}

\subsection{Fourth main result}

\begin{thm}
\label{4 thm}
Let $\bla$ be a partition on $n$ with at most $N$ parts. Then

\begin{enumerate}
\item[(i)]
The map \,$\tilde \mu_\bla\!: G_{k,s} \mapsto \tilde C^\bla_{k,s}$, \,$k=1\lc N$, \,$s>0$,
extends uniquely to an isomorphism
\,$\tilde\mu_\bla\!:\O_\bla\<\to\Bo\bigl(\DLs\bigr)$ of filtered
algebras. The isomorphism \,$\tilde \mu_\bla$ becomes an isomorphism of the
\,$\C[\si_1\lc\si_n]$-module \,$\O_\bla$ and the \,$\Czs$-module
$\Bo\bigl(\DLs\bigr)$ if we identify the algebras
\,$\C[\si_1\lc\si_n]$ and \,$\Czs$ by the map
\;$\si_s\mapsto \si_s(\bs z)$, \,$s=1\lc n$.

\vsk.2>
\item[(ii)]
The map \;$\tilde\nu_\bla: \O_\bla \to\>\DLs$\,,
\,$f\mapsto\tilde\mu_\bla(f)\,v_\bla^A$, is an isomorphism of filtered vector
spaces decreasing the index of filtration
by \,$\sum_{1\le i<j\le N}\la_i\la_j$. The isomorphism
\,$\tilde\nu_\bla$ identifies the \;$\Bo\bigl(\DLs\bigr)$-module \;$\DLs$
and the regular representation of the algebra \,$\O_\bla$.
\end{enumerate}
\end{thm}

\begin{proof}
The proof of Theorem \ref{4 thm} word by word coincides with that of
Theorem \ref{2 thm}.
\end{proof}

\begin{cor}
\label{cor on iso2}
The $\B^{\qq=\bs 1}$-modules $(\V^S)^\sing_\bla$ and $\DLs$ are isomorphic.
\end{cor}

\begin{proof} The corollary follows from Theorems \ref{2 thm} and \ref{4 thm}.
\end{proof}

\end{document}